\documentclass[reqno,A4paper]{amsart}

\usepackage{amssymb,amsfonts,amsthm,amsmath,amscd,relsize}
\usepackage{hyperref}
\usepackage{enumerate}

\usepackage{pgf,tikz,pgfplots}
\usetikzlibrary{calc}

\definecolor{qqqqff}{rgb}{0,0,1}
\definecolor{qqwuqq}{rgb}{0,0.39215686274509803,0}

\def\N{\mathbb N}
\def\Z{\mathbb Z}
\def\Q{\mathbb Q}

\def\a{\alpha}
\def\g{q}
\def\v{\varphi}

\theoremstyle{plain}
\newtheorem{theorem}{Theorem}

\newtheorem{lemma}{Lemma}
\newtheorem{corollary}{Corollary}
\newtheorem*{RT}{Rosenberger's Theorem}

\theoremstyle{definition}
\newtheorem{definition}{Definition}

\theoremstyle{remark}
\newtheorem{remark}{Remark}
\newtheorem{example}{Example}
\newtheorem*{notation}{Notation}

\begin{document}

\title{Equable Parallelograms on the Eisenstein Lattice}

\author{Christian Aebi* \and Grant Cairns**}

\newcommand{\acr}{\newline\indent}

\address{\llap{*\,}Coll\`ege Calvin\acr
Geneva\acr
Switzerland 1211}

\email{christian.aebi@edu.ge.ch}

\address{\llap{**\,}Department of Mathematical and Physical Sciences\acr
La Trobe University\acr
Melbourne\acr
Australia 3086}

\email{G.Cairns@latrobe.edu.au}


\begin{abstract}
This paper studies equable parallelograms whose vertices lie on the Eisenstein lattice. Using Rosenberger's Theorem on generalised Markov equations, we show that the set of these parallelograms forms naturally an infinite tree, all of whose vertices  have degree 4, bar the root which has degree 3. This study naturally complements the authors' previous study of 
equable parallelograms whose vertices lie on the integer lattice.
\end{abstract}

\maketitle

\section{Introduction} 
A planar polygon is said to be \emph{equable} if its perimeter equals its area. 
 
\begin{definition}
An \emph{Eisenstein lattice equable parallelogram} (or \emph{ELEP}, for short) is an equable parallelogram whose 
vertices lie on the Eisenstein  lattice $\Z[\omega]$, where $\omega=-\frac12+i\frac{\sqrt3}2$. 
\end{definition}

In \cite{ACLEQI,ACLEQII,ACLEQIII}, we investigated equable quadrilaterals with vertices on the integer  lattice.
This present paper begins a project of replicating this investigation for the Eisenstein lattice, with the goal of comparing the results.
In particular, in \cite{ACLEQI} we studied equable parallelograms with vertices on the integer lattice. This present paper follows the general approach adopted in \cite{ACLEQI}, but it can be read independently of \cite{ACLEQI}. We find that while the mathematics in this paper is very similar to that of \cite{ACLEQI}, the results are somewhat simpler. We saw in  \cite{ACLEQI} that the equable parallelograms with vertices on the integer  lattice form a forest of three trees, corresponding to three of Rosenberger's six generalized Markov equations; these are equations M, R1 and R3 in the notation of \cite{BU}. In the present work, we show that the ELEPs form a single tree corresponding to the Rosenberger equation R2. 
It is known (see \cite{Ro}) that the other two equations, R4 and R5, do not have coprime solutions and so do not appear in the type of study we undertake here or in \cite{ACLEQI}. So, in that sense, the tree of ELEPs forms the case that was curiously missing  in \cite{ACLEQI}. 

We should also remark that it is not surprising that there are relatively fewer equable parallelograms on the Eisenstein lattice than there are on the 
integer  lattice. It was already observed that for triangles, up to Euclidean transformations, there are only two equable triangles on the Eisenstein lattice, while there are five on the integer  lattice; see \cite{ACMon,ACetel,AClte} and the Appendix in \cite{ACLEQI}.

Let us now describe our main results. We first show in Lemma~\ref{L:bcint} below that the sides of an ELEP are necessarily of the form $n\sqrt3$ with $n\in\N$. Throughout this paper we will denote the side lengths $a\sqrt3$ and $b\sqrt3$. Note that an ELEP is completely determined, up to a Euclidean motion, by the integers $a,b$. Indeed, if $\theta$ denotes one of the angles between the sides, then the area is $3ab\sin\theta$ and so by equability, 
$\sin\theta=2(a+b)/\sqrt3ab$, which is determined by $a$ and $b$. Notice, incidentally, that from $\sin\theta=2(a+b)/\sqrt3ab$, there is no ELEP for which $\sin\theta$ is rational. In particular, there is no rectangular ELEP, nor any ELEP with angle $\pi/6$.

Our main aim in this paper is to study the values of $a,b$ for which an ELEP exists with sides $a\sqrt3,b\sqrt3$. 
In Section \ref{S:suff} we prove the following criterion, which is the exact analogue of \cite[Theorem~1]{ACLEQI} in that it can be rephrased as follows: the square of the product of the sides minus the square of the perimeter is a square.

\begin{theorem}\label{T:suff}
Given positive integers $a,b$, an Eisenstein lattice equable parallelogram with sides $a\sqrt3,b\sqrt3$ exists if and only if  $9a^2 b^2 -12(a+b)^2$ is a square.
\end{theorem}

\begin{corollary}\label{C:norhombus} There is no Eisenstein lattice equable rhombus.
\end{corollary}

In Section \ref{S:345} we use Rosenberger's Theorem on generalised Markov equations to prove  the following result.

\begin{theorem}\label{T:345}
 The set $\mathcal T$ of ordered pairs $(a,b)$ of positive integers $a< b$, for which $9a^2 b^2 -12(a+b)^2$ is a square, is the set
\begin{align*}
\mathcal T =\{2(q, r) \in\N^2 \ | \ &\text{there exist unique odd, coprime, positive integers}\  s,t \\
&\text{such that}\ \
s^2 + 3t^2 + 2q^2 = 6stq
\ \ \text{and}\ \  q \le r =3st-q\}.
\end{align*}
Furthermore, if $(a,b)=2(q, r)\in \mathcal T$, then $ab= 2(s^2+3t^2)$ and $a+b=6st$. Moreover, $q,r$ are coprime; in particular, 
$3\nmid ab$.
\end{theorem}

Theorems \ref{T:suff} and \ref{T:345} provide a classification of ELEPs up to Euclidean motions. 
In Section \ref{S:forest} we describe how all ELEPs can be derived from a root example, $(a,b)=(2,4)$, by successive applications of four functions. This gives the tree of ELEPs; see Figure \ref{F:tree}.

\begin{figure}[h]
\begin{tikzpicture}[scale=.6]
\def\r{1.732};

\draw [fill, blue!20] (0,0)
   -- (0, -4*\r/2)  
  -- (6,0) 
  -- (6,4*\r/2) 
 -- cycle;

\draw [thick, blue] (0,0)
   -- (0, -4*\r/2)  
  -- (6,0) 
  -- (6,4*\r/2) 
  -- cycle;

\foreach \i in {0,...,6}
\foreach \j in {-2,...,2}
\draw (\i,\j*\r) circle (.06);
\foreach \i in {0,...,7}
\foreach \j in {-2,...,1}
\draw (\i-1/2,\r*\j+\r/2) circle (.06);

 \draw [->] (0,-4) -- (0,4.3);
 \draw [->] (-1,-2*\r) -- (7,-2*\r);

  \end{tikzpicture}
\caption{The root ELEP, $(a,b)=(2,4)$}\label{F:root}
\end{figure}
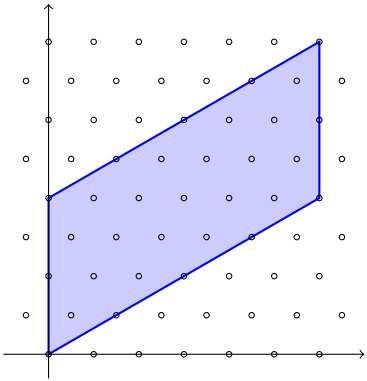

The paper \cite{ACLEQI} included a study of equable parallelograms, with vertices on the integer lattice, that have a pair of horizontal sides.
No such parallelogram exists on the Eisenstein lattice, because of Lemma~\ref{L:bcint}. 
Instead, in Section~\ref{S:horh},  we classify the ELEPs that have a horizontal diagonal. We show in Section~\ref{S:vertd} that there is no 
ELEP with a vertical diagonal. Finally, in Section~\ref{S:verts}, we classify the ELEPs with a vertical side, which turns out to be the same as the class of  ELEPs with a side of length $2\sqrt3$ or $4\sqrt3$. 

\begin{notation} In this paper, we employ the term \emph{positive} in the strict sense. So $\N=\{n\in\Z \ |\ n>0\}$.
\end{notation}


\section{Proof of Theorem \ref{T:suff}}\label{S:suff}

Recall that the square of the distance between any pair of points in the Eisenstein lattice is an integer. Indeed, if $z=x+y\omega\in \Z[\omega]$, then the square of the distance from the origin to $z$ is $x^2-xy+y^2$. 
If a triangle has its vertices on the Eisenstein lattice, then its area is of the form $\frac{\sqrt3}4 n$, where $n\in\N$. Indeed,
for the triangle with vertices $0,x+y\omega,z+w\omega$, with $x,y,z,w\in\Z$, the area is $\frac{\sqrt3}4 \vert xw-yz\vert $.

The following result was proved for triangles in \cite{ACetel}.

\begin{lemma}\label{L:bcint}
If $P$ is an equable polygon with vertices in $\Z[\omega]$, then the side lengths of $P$ are each of the form $\sqrt3n$, for some $n\in \N$.\end{lemma}

\begin{proof}
Since the area of any triangle with  vertices in $\Z[\omega]$ is of the form $\frac{\sqrt3}4n$, for some $n\in \N$,  the same is true for the area of $P$. Suppose $P$ has sides $s_1,\dots,s_n$, whose squares are therefore integers. The equable hypothesis gives  $\sum_{i=1}^n s_i\in\frac{\sqrt3}4\N$, which implies
 \[
 \sqrt{3s_1^2}+\dots+\sqrt{3s_n^2}=\sqrt{3}\sum_{i=1}^n s_i \in\Q.
 \]
But it is well known that if $\sum_{i=1}^n\sqrt{m_i}$ is rational for integers $m_1,\dots,m_n$, then $\sqrt{m_i}$ is rational for each $i$; see for example \cite{art} or \cite{Yuan}. Thus $\sqrt{3s_i^2}$ is rational for each $i$. Hence, as $3s_1^2,\dots,3s_n^2$ are integers, it follows that $\sqrt{3s_1^2},\dots,\sqrt{3s_n^2}$ are also integers. So the side lengths  are each of the form $\frac{n}{\sqrt3}$, for some $n\in \N$. Thus, since the squares of the side lengths are integers, the required result follows.
 \end{proof}

\begin{lemma}\label{L:diag}
Suppose an ELEP $P$ has  sides $a\sqrt3, b\sqrt3$, where $a,b\in\N$. Then $a+b\equiv 0 \pmod3$ and the lengths of the diagonals of $P$ are given by the following formula:
\[
d^2=3(a^2  +  b^2)\pm 2\sqrt{9a^2 b^2 -12(a+b)^2}.
\]
In particular,  $9a^2 b^2 -12(a+b)^2$ is a square.
\end{lemma}

\begin{proof}
Consider a diagonal of length $d$, so $d^2\in\N$.
By Heron's formula 
 \cite[Chap.~6.7]{OW}, the triangle with sides $a\sqrt3, b\sqrt3,d$  has area
\[
\sqrt{s(s-a\sqrt3)(s-b\sqrt3)(s-d)} ,
\]
where $s=\frac{\sqrt3a+\sqrt3b+d}2$ is the semi-perimeter.
Hence the equability hypothesis is 
\[
(\sqrt3a+\sqrt3b+d)(-\sqrt3a+\sqrt3b+d)(\sqrt3a-\sqrt3b+d)(\sqrt3a+\sqrt3b-d)=48(a+b)^2.
\]
 Expanding and rearranging the left hand side gives 
 \begin{equation}\label{E:abd}
  -9 (a^2  -  b^2)^2 + 6 (a^2 +b^2)d^2  -d^4=48(a+b)^2.
 \end{equation}
 In particular, the integer $d^2$ is divisible by 3, say $d^2=3D$. Then \eqref{E:abd} gives
\[ -3 (a^2  -  b^2)^2 + 6 (a^2 +b^2)D -3D^2=16(a+b)^2,
\]
and hence $a+b\equiv 0 \pmod3$, as required.
Furthermore, solving \eqref{E:abd} for $d^2$ gives
\[
d^2=3(a^2  +  b^2)\pm 2\sqrt{9a^2 b^2 -12(a+b)^2},
\]
as required. In particular, as $a,b,d^2$ are integers, $9a^2 b^2 -12(a+b)^2$ is a square.
\end{proof}

\begin{remark}
Suppose an ELEP $P$ has  sides $a\sqrt3,b\sqrt3$ and diagonals $d_1,d_2$. Then the above lemma gives
$d_1d_2=(a  +  b)\sqrt{48+9(a-b)^2}$. Indeed, by Lemma~\ref{L:diag},
\begin{align}
(d_1d_2)^2&=9(a^2  +  b^2)^2- 36a^2 b^2 +48(a+b)^2\notag\\
&=9(a^2  -  b^2)^2+48(a+b)^2=(a  +  b)^2(48+9(a-b)^2)\label{E:Bret},
\end{align}
as required.
We remark that the same formula can be deduced from Bretschneider's formula and the equable hypothesis, without the use of Lemma~\ref{L:diag}.
Notice that the 3-adic order, $\nu_3((a  +  b)^2(48+9(a-b)^2))$, of the right-hand-side of \eqref{E:Bret} is odd. Hence 
$\nu_3(d_1^2)\not\equiv \nu_3(d_2^2) \pmod 2$.
\end{remark}

\begin{proof}[Proof of Theorem \ref{T:suff}]
The necessity of the condition was shown in Lemma \ref{L:diag}. Therefore, assume that $a,b$ are positive integers such that $9a^2 b^2 -12(a+b)^2$ is a square.
Consider a triangle $T$ with sides $a\sqrt3,b\sqrt3$ and 
\[
d:=\sqrt{3(a^2  +  b^2)+ 2\sqrt{9a^2 b^2 -12(a+b)^2}}.
\]
Notice that such a triangle exists because $a\sqrt3,b\sqrt3 < d <(a+b)\sqrt3$, where the latter inequality holds as 
$3(a^2  +  b^2)+ 2\sqrt{9a^2 b^2 -12(a+b)^2} <3(a+b)^2$ since $2\sqrt{9a^2 b^2 -12(a+b)^2} <6ab$.
Let $\theta$ denote the angle between sides of length $a\sqrt3,b\sqrt3$ and note that $\theta$ is obtuse since $d^2\ge 3a^2+3b^2$. Therefore
\[
d^2=3a^2+3b^2-6ab\cos\theta=3a^2+3b^2+2\sqrt{9a^2b^2-9a^2b^2\sin^2\theta}.
\]
So, from the definition of $d$, we have $9a^2b^2\sin^2\theta=12 (a+b)^2$, thus $3ab\sin\theta= 2\sqrt3(a+b)$. But since  twice the area of $T$ is $3ab\sin\theta$, the area of $T$ is then $\sqrt3(a+b)$. Now consider the parallelogram $P$ made from two copies of $T$. From what we have just seen, $P$ is equable. It remains to show that $P$ can be realized on the Eisenstein lattice, or equivalently, that  $T$ can be realized on the Eisenstein lattice. To do this, we employ the following result.

\begin{theorem}[\cite{ACMon}]\label{T:Monthly}
A planar triangle $T$  is realizable on the Eisenstein lattice if and only if the following three conditions hold:
\begin{enumerate}
\item
the area of $T$ is of the form $\frac{\sqrt3}4 n$, where $n\in\N$, 
\item the squares of the side lengths of $T$ are integers,
\item one of the side lengths of $T$  is of the form $r\sqrt{t}$, where $r, t\in\N$  and  $t$ has no prime divisors congruent to $2 \pmod 3$.
\end{enumerate}
\end{theorem}

Note that firstly, we saw above that  the area of $T$ is $\sqrt3(a+b)\in\sqrt3\N$. Secondly, the squares of the side lengths are $3a,3b,3(a^2  +  b^2)+ 2\sqrt{9a^2 b^2 -12(a+b)^2}$ which are all integers. Thirdly, the length $\sqrt3a$ has the form  $r\sqrt{t}$, where $r, t\in\N$  and  $t$ has no prime divisors congruent to $2 \pmod 3$. So by Theorem~\ref{T:Monthly}, $T$ has a realization on the Eisenstein lattice.
\end{proof}

\begin{proof}[Proof of Corollary \ref{C:norhombus}]
Suppose we had an Eisenstein lattice equable rhombus with side length $a$.
Then by Lemma~\ref{L:diag}, $9a^4 -48a^2$ is a square and hence $9a^2 -48$ is a square. But this is impossible as $6$ is not a quadratic residue modulo $9$.
\end{proof}


\section{Proof of Theorem \ref{T:345}}\label{S:345}

The aim of this section is to prove Theorem \ref{T:345}. We first show that the elements of the set $\mathcal T$ have the required property.
  Suppose that $(a,b)=2(q, r)\in \mathcal T$, where for coprime positive integers $s,t$ we have
 $ s^2 + 3t^2 + 2q^2 = 6stq$ and $q <r =3st-q$. 
  Hence $q,r$ are the two solutions of the quadratic equation
\begin{equation}\label{E:added}
 2u^2-6stu+ (s^2+3t^2)=0,
\end{equation}
 in $u$, and so $q+r=3st$ and $2qr=s^2+3t^2$. Thus we have
  \begin{align*}
9a^2b^2-12(a+b)^2&=4(9 \cdot4q^2r^2-12(q+r)^2)\\
&=4(9(s^2+3t^2)^2-12\cdot 3^2s^2t^2)=4\cdot 9(s^2-3t^2)^2,
 \end{align*}
which is a square, as required.

We now show that every solution is in $\mathcal T$.
The main tool we use  is Rosenberger's Theorem on generalised Markov equations. Recall that in \cite{Ro} Rosenberger considered equations of the form 
\begin{equation}\label{E:Ro}
ax^2+by^2+cz^2=dxyz,
\end{equation}
where $a,b,c$ are pairwise coprime positive integers with $a\le b\le c$ such that $a,b,c$ all divide $d$. We are only interested in positive integer solutions, that is, $x,y,z\in\N$, so we use the word \emph{solution} to mean positive integer solution. Rosenberger's remarkable result is that only 6 such equations have a solution and when such a solution exists, there are infinitely many solutions. We use the R1--R5 notation of \cite{BU}.

\begin{RT}[\cite{Ro}]
Equation \eqref{E:Ro} only has a solution in the following 6 cases:
\begin{enumerate}
\item[\textup{M:}] $x^2 + y^2 + z^2 = 3xyz$ (Markov's equation),
\item[\textup{R1:}] $x^2 + y^2 + 2z^2 = 4xyz$,
\item[\textup{R2:}]  $x^2 + 2y^2 + 3z^2 = 6xyz$,
\item[\textup{R3:}]  $x^2 + y^2 + 5z^2 = 5xyz$,
\item[\textup{R4:}]  $x^2 + y^2 + z^2 = xyz$,
\item[\textup{R5:}]  $x^2 + y^2 + 2z^2 = 2xyz$.
\end{enumerate}
\end{RT}

We will also require  the following classical result.  There are several proofs of this result; a direct, elementary proof is given in \cite{AD}.

\begin{lemma}\label{L:trips} Suppose that positive integers $x,y,z$ satisfy the equation $x^2+3y^2=z^2$. Then there exist $k\in\N$ and coprime $s,t\in\N$ such that 
\[
x=\frac{k}2\, |s^2-3t^2|,\ y=kst,\ z=\frac{k}2(s^2+3t^2),
\]
where $k$ is even if $s,t$ have different parity.
\end{lemma}

Returning to the proof of Theorem \ref{T:345}, suppose that $a,b$ are positive integers with $a< b$ such that $9a^2 b^2 -12(a+b)^2$ is a square. So 
$a+b\equiv 0\pmod 3$ and $a^2 b^2 -12\left(\frac{a+b}3\right)^2$ is a square integer.
Let $z=ab,y=\frac{2(a+b)}3,x=\sqrt{a^2 b^2 -12\left(\frac{a+b}3\right)^2}$, so that
$x^2+3y^2=z^2$, and $y$ is even. So by Lemma~\ref{L:trips}, there exist $k'\in\N$ and coprime $s,t\in\N$ such that 
\[
x=\frac{k'}2\, |s^2-3t^2|,\ y=k'st,\ z=\frac{k'}2(s^2+3t^2),
\]
where $k'$ is even if $s,t$ have different parity. In particular, from $z$ and $y$, we have
\begin{equation}\label{E:py1}
ab= \frac{k'}2(s^2+3t^2), \qquad a+b=\frac{3k'}2st.
\end{equation}
Since $s,t$ are coprime,  $s,t$ are both odd if they have the same parity. But in this case, $k'$ must be even, by the second Equation in \eqref{E:py1}.
So  $k'$ is even in all cases; set $k'=2k$. Then from Equation~\eqref{E:py1}, we have
\begin{equation}\label{E:py2}
ab= k(s^2+3t^2), \qquad a+b=3kst.
\end{equation}
Hence $a,b$ are solutions to the quadratic equation
$X^2-3kst X+k(s^2+3t^2)=0$.
In particular, we have
\begin{equation}\label{E:xb2}
ks^2+3kt^2+a^2=3kast.
\end{equation}
Let $k=fg^2$, where $f$ is square-free.
From \eqref{E:xb2}, $f$ divides $a^2$ and hence $f$ divides $a$. Let $a=f\a$. Dividing \eqref{E:xb2} by $f$ gives
\begin{equation}\label{E:xb3}
g^2s^2+3g^2t^2+f\a^2=3fg^2 st\a.
\end{equation}
From \eqref{E:xb3}, $g^2$ divides $f\a^2$ and hence $g^2$ divides $\a^2$, and thus $g$ divides $\a$. Let $\a=g\g$.  Dividing \eqref{E:xb3} by $g^2$ gives
\begin{equation}\label{E:xb4}
s^2+3t^2+f\g^2=3fg\, st\g.
\end{equation}
Thus by Rosenberger's Theorem, using $x=s,y=\g,z=t$, there is only one possibility for \eqref{E:xb4} to have a solution, namely it is the  Equation~R2, with $f=2$ and $g=1$. Consequently, we have $k=2$ and $ a=2\g$. So Equation~\eqref{E:xb4} gives
\begin{equation}\label{E:xb5}
s^2+3t^2+2\g^2=6 st\g.
\end{equation}
In particular, $a$ is even and hence $b$ is even, say $b=2r$, by the second Equation in \eqref{E:py2}.

Notice that as $s^2 + 3 t^2 + 2 \g^2 = 6 s t \g$, then modulo $2$, this gives $s+t\equiv 0$, and so $s,t$ are both odd.

We now show that if $(a,b)=(2q,2r)$, then $q,r$ are coprime.
Rewriting 
\eqref{E:py2}, we have
\begin{equation}\label{E:py3}
ab= 2(s^2+3t^2), \qquad a+b=6st.
\end{equation}
Suppose $p$ is an odd prime divisor of $\gcd(a,b)$. Then $p^2$ divides $s^2+3t^2$ and $p$ divides $3st$, by Equation~\eqref{E:py3}. Thus $p$ divides $(s+3t)^2$ and $(s-3t)^2$, and so $p$ divides $s+3t$ and $s-3t$. Hence $p$ divides $2s$ and $6t$. Consequently, as $p$ is odd and $s,t$ are coprime, $p=3$ and furthermore,  $3$ divides $s$ and $3$ doesn't divide $t$.
Since $3$ divides $a$ and $b$, we have $9$ divides $s^2+3t^2$,  from the first Equation in \eqref{E:py3}. But this is impossible since $9$ divides $s^2$ and $3$ doesn't divide $t$. Hence $\gcd(a,b)$ is a power of $2$. Then, as $s,t$ are both odd, the second Equation in \eqref{E:py3} gives $\gcd(a,b)=2$, as required. 

Since $a+b=6st$ and $\gcd(a,b)=2$, neither $a$ nor $b$ is divisible by $3$. It remains to show that, given $a,b$, the integers  $s,t$ are unique.
Suppose that \eqref{E:py3} holds and that one has coprime positive integers $s',t'$ with 
\begin{equation}\label{E:py32}
ab= 2(s'^2+3t'^2), \qquad a+b=6s't'.
\end{equation}
 From the second equations in \eqref{E:py3} and \eqref{E:py32}, we have $s't'=s t$. From \eqref{E:py32}, we also have 
$s'^2 + 3 t'^2 + 2 \g^2 = 6 s' t' \g$,
where $a=2\g$, as before. Hence, by \eqref{E:py3} and \eqref{E:xb5}, $s'^2 + 3 t'^2 =s^2 + 3 t^2 $, so as $t'=st/s'$,
\[
s'^4 + 3 s^2t^2 =s'^2(s^2 + 3 t^2),
\]
and hence $(s'^2-s^2)(s'^2-3t^2)=0$. As $s'$ is a positive integer, it follows that $s'=s$ and thus $t'=t$.
This completes the proof of Theorem~\ref{T:345}.\hfill\qed

\begin{remark}\label{R:sdiv3} Since $a=2q$ is not divisible by $3$, neither is $q$.
So by \eqref{E:xb5}, $s$ is not divisible by $3$. 
\end{remark}


\section{Comments on Theorem \ref{T:345}}\label{S:345com}

Theorem \ref{T:345} identifies the set $\mathcal T$ of ordered pairs $(a,b)$ of possible side lengths of ELEPs, divided by $\sqrt3$, with $a< b$ and $a,b$ even; $a=2q$ and $b=2r$. It relates these pairs $(a,b)$ to pairs of coprime positive integers $s,t$ for which the equation
\[
s^2+3t^2+2u^2=6 stu
\]
is satisfied for both $u=q$ and $u=r$.
From given $s,t$, we have, from \eqref{E:added}, 
\begin{equation}\label{E:ab}
a=3st-\sqrt{9s^2t^2-2(s^2+3t^2)},\qquad b=3st+\sqrt{9s^2t^2-2(s^2+3t^2)}.
\end{equation}
Conversely, given $a,b$,  we have, from \eqref{E:py3}, 
 \[
s^2+3t^2=\frac{ab}2\ \text{and}\ s^2\cdot 3t^2=\frac{(a+b)^2}{12}
\]
and so
\begin{align}
s^2&=\frac{3a b +\sigma(a,b)  \sqrt{9a^2 b^2-12(a+b)^2}}{12},\label{E:s}\\
3t^2&=\frac{3a b -\sigma(a,b)  \sqrt{9a^2 b^2-12(a+b)^2}}{12},\label{E:t}
\end{align}
where $\sigma(a,b)=\pm 1$, and will now be explained. By Remark~\ref{R:sdiv3}, $s\not\equiv 0\pmod3$, so $s^2\equiv 1\pmod3$.
Notice that
\[
s^2=\frac{3a b +\sigma(a,b)  \sqrt{9a^2 b^2-12(a+b)^2}}{12}=qr +\sigma(a,b)  \sqrt{q^2 r^2-\frac{(q+r)^2}3}.
\]
By Theorem~\ref{T:345}, $a,b$ are not divisible by $3$. So, as $q+r=3st$, we have $qr\equiv -1\pmod3$. Hence,  since $\sigma(a,b)=1$ if and only if $s^2>3t^2$, we have the following result.

\begin{lemma}\label{L:sigma} Suppose $(a,b)\in \mathcal T$ with $a=2q$ and $b=2r$. Then 
\[
\sigma(a,b)\equiv -  \sqrt{q^2 r^2-\frac{(q+r)^2}3}\pmod3,
\]
and the following conditions are equivalent:
\begin{enumerate}
\item[\textup{(a)}] $\sigma(a,b)=1$,
\item[\textup{(b)}] $s^2>3t^2$,
\item[\textup{(c)}] $\sqrt{q^2 r^2-\frac{(q+r)^2}3} \equiv -1\pmod3$.
\end{enumerate}
\end{lemma}

\begin{remark}\label{R:integers}
Suppose that an element $z=x'+y'\omega\in\Z[\omega]$ has length $\sqrt3n$, for some $n\in \N$. Then in complex numbers,
$z=x'-\frac{y'}2+\frac{\sqrt3y'}2i$ and
\[
3n^2= x'^2-x'y'+y'^2.
\]
Thus, if $n$ is even, then $x',y'$ are necessarily both even. Let $y'=2y$ and $x=x'-y$, so $z=x+y\sqrt3i$, where $x,y\in\Z$. 
In particular, this is the case for the sides of an ELEP, by Theorem~\ref{T:345}.
\end{remark}


\section{The tree of ELEPs}\label{S:forest}

Consider the set $\mathcal S$ of solutions $(s,t,u)$, with $s,t$ coprime, of the  Markov-Rosenberger equation given in 
\eqref{E:xb5} with $\g=u$:
\begin{equation}\label{E:u}
s^2+3t^2+2u^2=6 stu.
\end{equation}
Note that we are not assuming that $u<3st$. Following the presentation given in \cite{BU}, from a solution $x = (s,t,u)$ to 
\eqref{E:xb5}, one can generate three new solutions by
applying the involutions:
\begin{align*}
\phi_1(x)&= (6 tu-s,t,u)\\
\phi_2(x)&= (s,2su-t,u)\\
\phi_3(x)&= (s,t,3st-u) .
\end{align*}
The group of transformations of $\mathcal S$ generated by the maps $\phi_i$ is the free product of three copies of $\Z_2$, and this group acts transitively on $\mathcal S$.
Moreover, the  maps $\phi_i$ give the set $\mathcal S$ of solutions the structure of an infinite  binary tree: each solution is a vertex and two distinct solutions are connected by an edge if  one of the  maps $\phi_i$ sends one solution to the other.
The \emph{fundamental solution} has the smallest values of $s+t+u$; it is $(s,t,u)=(1,1,1)$. 

\begin{lemma}\label{L:fix} The fixed point sets of the maps $\phi_1,\phi_3$ are empty, and $(1,1,1)$ is the unique  fixed point of $\phi_2$.
\end{lemma}

\begin{proof}
If $(s,t,u)$ is a fixed point of $\phi_1$, then 
$s=3tu$. Hence $t=1$, as $s, t$ are coprime. Replacing in \eqref{E:u} gives the contradiction $3=7u^2$.
Similarly,
a fixed point $(s,t,u)$ of $\phi_3$ would have $2u=3st$, which is impossible as $s,t$ are odd.
If $(s,t,u)$ is a fixed point of $\phi_2$, we have
$t=su$, so $s=1$, as $s, t$ are coprime. Then \eqref{E:u} gives $u^2=1$ and so $t^2=1$. Thus finally $s=t=u=1$.
\end{proof}

In summary so far: the group $\Z_2\ast\Z_2\ast \Z_2$ generated by $\phi_1,\phi_2,\phi_3$ acts freely on $\mathcal S$ except at the fundamental solution $(s,t,u)=(1,1,1)$, which we can take as the root of the tree $\mathcal S$.

Having recalled Rosenberger's theory, we now describe how the solution tree $\mathcal S$ of the Markov-Rosenberger equation determine the induced structure on the set $\mathcal T$ of  ELEPs. 
Motivated by
\eqref{E:ab}, we define the map $\pi:\mathcal S \to \mathcal T$ by
\begin{equation}\label{E:pi}
\pi (s,t,u)=\left(3st-\sqrt{9s^2t^2-2(s^2+3t^2)},\ 3st+\sqrt{9s^2t^2-2(s^2+3t^2)}\right).
\end{equation}
The map $\pi$ is well defined and surjective by Theorem \ref{T:345}. Note that if $\pi (s,t,u)=(a,b)$, then by \eqref{E:ab} and \eqref{E:u}, either $a=2u$ or $b=2u$. Furthermore, trivially, $\pi \circ \phi_3(s,t,u)=(a,b)$; that is, $(s,t,u)$ and $ \phi_3(s,t,u)$ correspond to the same ELEP.
In other words, for each $(a,b)\in \mathcal T$, the pre-image $\pi^{-1}(a,b) $ consist of two points, which are interchanged by the involution $\phi_3$.

Consequently, as contraction of edges of a tree produces another tree,  we can form a tree structure on $\mathcal T$ by contracting each of the edges of $\mathcal S$ that are given by the map $ \phi_3$; see Figure \ref{F:contract}. This contraction turns vertices of degree 3 into vertices of degree 4. This is the case for all vertices except for the image of the fundamental solution, which has degree 3. This is the \emph{root} $(2,4)\in \mathcal T$. 
 The tree $\mathcal T$  is shown in Figure \ref{F:tree}. Here the elements $(a,b)$ are shown above the corresponding pairs $(s,t)$.

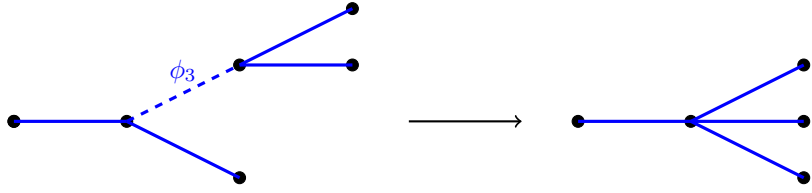
\begin{figure}[h]
\begin{tikzpicture}[scale=1.5][line cap=round,line join=round,>=triangle 45,x=1cm,y=.5cm]
\def\y{.5};
\def\d{5};
\draw [fill=black] (-1,0) circle (1.5pt);
\draw [fill=black] (0,0) circle (1.5pt);
\draw [fill=black] (1,\y) circle (1.5pt);
\draw [fill=black] (1,-\y) circle (1.5pt);
\draw [fill=black] (-1,0) circle (1.5pt);
\draw [fill=black] (2,2*\y) circle (1.5pt);
\draw [fill=black] (2,\y) circle (1.5pt);
\draw [line width=1.2pt,dashed,color=qqqqff] (0,0)-- (1,\y) node[midway,above] {$\phi_3$};
\draw [line width=1.2pt,color=qqqqff] (0,0)-- (1,-\y) ;
\draw [line width=1.2pt,color=qqqqff] (0,0)-- (-1,0) ;
\draw [line width=1.2pt,color=qqqqff] (1,\y)-- (2,2*\y) ;
\draw [line width=1.2pt,color=qqqqff] (1,\y)-- (2,\y) ;
\draw [->,line width=0.8pt] (2.5,0) -- (3.5,0);
draw [fill=black] (\d-1,0) circle (1.5pt);
\draw [fill=black] (\d,0) circle (1.5pt);
\draw [fill=black] (\d+1,-\y) circle (1.5pt);
\draw [fill=black] (\d-1,0) circle (1.5pt);
\draw [fill=black] (\d+1,\y) circle (1.5pt);
\draw [fill=black] (\d+1,0) circle (1.5pt);
\draw [line width=1.2pt,color=qqqqff] (\d,0)-- (\d+1,-\y) ;
\draw [line width=1.2pt,color=qqqqff] (\d,0)-- (\d-1,0) ;
\draw [line width=1.2pt,color=qqqqff] (\d,0)-- (\d+1,\y) ;
\draw [line width=1.2pt,color=qqqqff] (\d,0)-- (\d+1,0) ;
\end{tikzpicture}
\caption{Contraction of the $\phi_3$ edges}\label{F:contract}
\end{figure}

We now see how the maps $\phi_1,\phi_2,\phi_3$ on $ \mathcal S$ generate  maps on $\mathcal T$.
For each $i=1,2$,  maps $\v_i,\psi_i: \mathcal T\to \mathcal T$ can be naturally defined as follows. If $(a,b)\in \mathcal T$, with $a=2q,b=2r$, let $(s,t,q)\in \mathcal S$ such that $\pi (s,t,q)=(a,b)$. Then set
\begin{equation}\label{E:vp}
\v_i(a,b):=\pi \circ \phi_i(s,t,q),\qquad \psi_i(a,b):=\pi \circ\phi_i(s,t,r).
\end{equation}
Note that from \eqref{E:py3} we have $a+b=6st$,
so $(a,b)=(2q, 6st-2q)$.
From the definition of $\phi_1$, the  map $\v_1$ leaves $a$ and $t$  unchanged and $s$ is changed to $s'= 6qt-s= 3at-s$. Then under $\v_1$, the value of $b$ is changed to
\begin{align}
b'&= 6s't-a=18at^2-6st-a=18at^2-(a+b)-a\notag \\
&=\frac{3a^2 b -\sigma(a,b)  a\sqrt{9a^2 b^2-12(a+b)^2}}2-2a-b, \label{E:b'}
\end{align}
using \eqref{E:t}. Note that a priori, we don't know which of the resulting components, $a$ or $b'$, is the larger.
So we set
\[
\v_1: (a,b) \mapsto \left(\min(a,b'),\max(a,b')\right).
\]
Similarly, using $\phi_2$, we set
\[
\v_2: (a,b) \mapsto \left(\min(a,b''),\max(a,b'')\right),
\]
where
\begin{equation}\label{E:b''}
b''= \frac{3a^2 b +\sigma(a,b)  a\sqrt{9a^2 b^2-12(a+b)^2}}2-2a-b. 
\end{equation}

Similarly, analogous to $\v_1,\v_2$, interchanging the roles of $a$ and $b$, we have two further maps:
\begin{align*}
\psi_1: (a,b) &\mapsto \left(\min(a',b),\max(a',b)\right),\\
\psi_2: (a,b) &\mapsto \left(\min(a'',b),\max(a'',b)\right),\end{align*}
where
\begin{align*}
 a'&=\frac{3a b^2 +\sigma(a,b)  b\sqrt{9a^2 b^2-12(a+b)^2}}2-a-2b,\\
 a''&=\frac{3a b^2 -\sigma(a,b)  b\sqrt{9a^2 b^2-12(a+b)^2}}2-a-2b.
\end{align*}
These maps are obtained from the maps $\phi_1\circ \phi_3$ and $\phi_2\circ \phi_3$, respectively.

Note that while $\phi_1,\phi_2,\phi_3$ are involutions and generate a group of transformations of $\mathcal S$, the same is not true of the maps $\v_1,\v_2,\psi_1,\psi_2$ of $\mathcal T$. Indeed, they are not even all bijections. For example,
$\phi_2(2,4)=(2,4)=\phi_2(4,14)$.
Moreover, the study of the maps $\phi_1,\phi_2,\psi_1,\psi_2$ is considerably complicated by the term $\sigma(a,b)$. So it is often convenient to work instead  with the pairs $(s,t)$. From the definition of the maps $\phi_i$, we have
\[
\phi_1(s,t,u/2)= (3 ut-s,t,u/2),\quad
\phi_2(s,t,u/2)= (s,us-t,u/2),
\]
where  $u=a$ or $b$, and $a,b$ are given by \eqref{E:ab}.
We define the linear involutions $\phi_{1,u},\phi_{2,u}$ on the set 
$\Sigma:=\{(s,t)\in\N^2: s,t\ \text{odd and coprime}\}$,
 by $\phi_{1,u}(s,t)= (3u t-s,t), 
\phi_{2,u}(s,t)= (s,us-t)$, or equivalently in matrix form,
\begin{equation}\label{E:phimatrices}
\phi_{1,u}= 
\begin{pmatrix}-1&3u\\0&1\end{pmatrix},\qquad
\phi_{2,u}=
\begin{pmatrix}1&0\\u&-1\end{pmatrix}.
\end{equation}
By taking $u=a$ or $b$, \eqref{E:phimatrices} effectively gives four functions. For  given $(a,b)$, it is often easier to employ these functions than the functions $\v_1,\v_2,\psi_1,\psi_2$, and then use 
\eqref{E:ab} to determine the resulting corresponding values of $a,b$, where necessary.
In Figure~\ref{F:tree}, and later in Figure~\ref{F:a4},  the edges are labelled with the corresponding maps $\phi_{i,u}$, for $i\in\{1,2\}$ and 
$u\in\{a,b\}$.

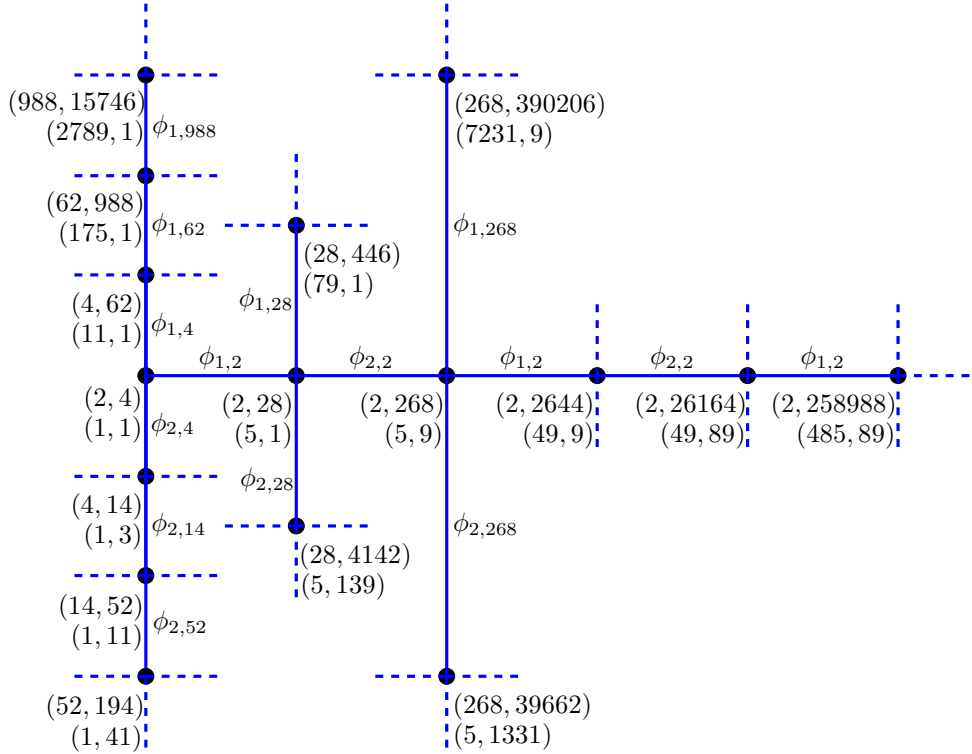
\begin{figure}[h]
\begin{tikzpicture}[scale=2][line cap=round,line join=round,>=triangle 45,x=1cm,y=1cm]
\def\fx{0};
\def\fy{0};
\def\rx{1};
\def\ry{0};
\def\rrx{2};
\def\rry{0};
\def\rrrx{3};
\def\rrry{0};
\def\rrrrx{4};
\def\rrrry{0};
\def\rrrrrx{5};
\def\rrrrry{0};
\def\rux{1};
\def\ruy{1};
\def\rdx{1};
\def\rdy{-1};
\def\rrux{2};
\def\rruy{2};
\def\rrdx{2};
\def\rrdy{-2};
\def\ux{0};
\def\uy{1};
\def\uuy{2};
\def\uuuy{3};
\def\dx{0};
\def\dy{-1};
\def\ddy{-2};
\def\dddy{-3};
\draw [fill=black] (\fx,\fy) circle (1.5pt);
\draw[color=black] (\fx-.22,\fy-.16)  node {$(2,4)$};
\draw[color=black] (\fx-.22,\fy-.36)  node {$(1,1)$};
\draw[color=black] (\fx+.5,\fy+.11)  node {$\phi_{1,2}$};
\draw [fill=black] (\rx,\ry) circle (1.5pt);
\draw[color=black] (\rx-.26,\ry-.2) node {$(2,28)$};
\draw[color=black] (\rx-.22,\ry-.4) node {$(5,1)$};
\draw[color=black] (\rx+.5,\ry+.11)  node {$\phi_{2,2}$};
\draw [fill=black] (\rrx,\rry) circle (1.5pt);
\draw[color=black] (\rrx-.3,\rry-.2) node {$(2,268)$};
\draw[color=black] (\rrx-.22,\rry-.4) node {$(5,9)$};
\draw[color=black] (\rrx+.5,\rry+.11)  node {$\phi_{1,2}$};
\draw [fill=black] (\rrrx,\rrry) circle (1.5pt);
\draw[color=black] (\rrrx-.35,\rrry-.2) node {$(2,2644)$};
\draw[color=black] (\rrrx-.26,\rrry-.4) node {$(49,9)$};
\draw[color=black] (\rrrx+.5,\rrry+.11)  node {$\phi_{2,2}$};
\draw [fill=black] (\rrrrx,\rrrry) circle (1.5pt);
\draw[color=black] (\rrrrx-.38,\rrrry-.2) node {$(2,26164)$};
\draw[color=black] (\rrrrx-.3,\rrrry-.4) node {$(49,89)$};
\draw[color=black] (\rrrrx+.5,\rrrry+.11)  node {$\phi_{1,2}$};
\draw [fill=black] (\rrrrrx,\rrrrry) circle (1.5pt);
\draw[color=black] (\rrrrrx-.43,\rrrrry-.2) node {$(2,258988)$};
\draw[color=black] (\rrrrrx-.35,\rrrrry-.4) node {$(485,89)$};
\draw [fill=black] (\rux,\ruy) circle (1.5pt);
\draw[color=black] (\rux+.37,\ruy-.2) node {$(28,446)$};
\draw[color=black] (\rux+.28,\ruy-.4) node {$(79,1)$};
\draw[color=black] (\rux-.2,\ruy-.5) node {$\phi_{1,28}$};
\draw [fill=black] (\rdx,\rdy) circle (1.5pt);
\draw[color=black] (\rdx+.39,\rdy-.2) node {$(28,4142)$};
\draw[color=black] (\rdx+.31,\rdy-.4) node {$(5,139)$};
\draw[color=black] (\rdx-.19,\rdy+.3) node {$\phi_{2,28}$};
\draw [fill=black] (\rrux,\rruy) circle (1.5pt);
\draw[color=black] (\rrux+.55,\rruy-.2) node {$(268,390206)$};
\draw[color=black] (\rrux+.38,\rruy-.4) node {$(7231,9)$};
\draw[color=black] (\rrux+.26,\rruy-1) node {$\phi_{1,268}$};
\draw [fill=black] (\rrdx,\rrdy) circle (1.5pt);
\draw[color=black] (\rrdx+.5,\rrdy-.2) node {$(268,39662)$};
\draw[color=black] (\rrdx+.37,\rrdy-.4) node {$(5,1331)$};
\draw[color=black] (\rrdx+.26,\rrdy+1) node {$\phi_{2,268}$};
\draw [fill=black] (\fx,\uy*.67) circle (1.5pt);
\draw[color=black] (\fx-.26,\uy*.67-.2) node {$(4,62)$};
\draw[color=black] (\fx-.26,\uy*.67-.4) node {$(11,1)$};
\draw[color=black] (\fx+.18,\uy*.33) node {$\phi_{1,4}$};
\draw[color=black] (\fx-.34,\uuy*.67-.2) node {$(62,988)$};
\draw[color=black] (\fx-.3,\uuy*.67-.4) node {$(175,1)$};
\draw[color=black] (\fx+.22,\uuy*.5) node {$\phi_{1,62}$};
\draw[color=black] (\fx-.46,\uuuy*.67-.2) node {$(988,15746)$};
\draw[color=black] (\fx-.34,\uuuy*.67-.4) node {$(2789,1)$};
\draw[color=black] (\fx+.26,\uuuy*.55) node {$\phi_{1,988}$};
\draw[color=black] (\fx-.26,\dy*.67-.2) node {$(4,14)$};
\draw[color=black] (\fx-.22,\dy*.67-.4) node {$(1,3)$};
\draw[color=black] (\fx+.18,\dy*.33) node {$\phi_{2,4}$};
\draw[color=black] (\fx-.3,\ddy*.67-.2) node {$(14,52)$};
\draw[color=black] (\fx-.26,\ddy*.67-.4) node {$(1,11)$};
\draw[color=black] (\fx+.22,\ddy*.5) node {$\phi_{2,14}$};
\draw[color=black] (\fx-.34,\dddy*.67-.2) node {$(52,194)$};
\draw[color=black] (\fx-.265,\dddy*.67-.4) node {$(1,41)$};
\draw[color=black] (\fx+.23,\dddy*.55) node {$\phi_{2,52}$};
\draw [fill=black] (\fx,\uy*1.33) circle (1.5pt);
\draw [fill=black] (\fx,\uy*2) circle (1.5pt);
\draw [fill=black] (\fx,\dy*.67) circle (1.5pt);
\draw [fill=black] (\fx,\dy*1.33) circle (1.5pt);
\draw [fill=black] (\fx,\dy*2) circle (1.5pt);
\draw [line width=1.2pt,color=qqqqff] (\fx,\fy)-- (\ux,\uy*2-.05);
\draw [line width=1.2pt,color=qqqqff] (\fx,\fy)-- (\ux,\uy*1.33-.05);
\draw [line width=1.2pt,color=qqqqff] (\fx,\fy)-- (\ux,\uy*0.67-.05);
\draw [line width=1.2pt,color=qqqqff] (\fx,\fy)-- (\dx,\dy*2+.05);
\draw [line width=1.2pt,color=qqqqff] (\fx,\fy)-- (\dx,\dy*1.33+.05);
\draw [line width=1.2pt,color=qqqqff] (\fx,\fy)-- (\dx,\dy*.67+.05);
\draw [line width=1.2pt,color=qqqqff] (\fx,\fy)-- (\rx,\ry);
\draw [line width=1.2pt,color=qqqqff] (\rx,\ry)-- (\rrx,\rry);
\draw [line width=1.2pt,color=qqqqff] (\rx,\ry)-- (\rux,\ruy);
\draw [line width=1.2pt,color=qqqqff] (\rx,\ry)-- (\rdx,\rdy);
\draw [line width=1.2pt,color=qqqqff] (\rrx,\rry)-- (\rrrx,\rrry);
\draw [line width=1.2pt,color=qqqqff] (\rrx,\rry)-- (\rrux,\rruy);
\draw [line width=1.2pt,color=qqqqff] (\rrx,\rry)-- (\rrdx,\rrdy);
\draw [line width=1.2pt,color=qqqqff] (\rrrx,\rrry)-- (\rrrrx,\rrrry);
\draw [line width=1.2pt,color=qqqqff] (\rrrrx,\rrrry)-- (\rrrrrx,\rrrrry);
\draw [line width=1.2pt,dashed,color=qqqqff] (\fx,\uy*2)-- (\fx,\uy*2+.5);
\draw [line width=1.2pt,dashed,color=qqqqff] (\fx,\uy*2)-- (\fx+.5,\uy*2);
\draw [line width=1.2pt,dashed,color=qqqqff] (\fx,\uy*2)-- (\fx-.5,\uy*2);
\draw [line width=1.2pt,dashed,color=qqqqff] (\fx,\uy*1.33)-- (\fx+.5,\uy*1.33);
\draw [line width=1.2pt,dashed,color=qqqqff] (\fx,\uy*1.33)-- (\fx-.5,\uy*1.33);
\draw [line width=1.2pt,dashed,color=qqqqff] (\fx,\uy*.67)-- (\fx+.5,\uy*.67);
\draw [line width=1.2pt,dashed,color=qqqqff] (\fx,\uy*.67)-- (\fx-.5,\uy*.67);
\draw [line width=1.2pt,dashed,color=qqqqff] (\fx,\dy*2)-- (\fx,\dy*2-.5);
\draw [line width=1.2pt,dashed,color=qqqqff] (\fx,\dy*2)-- (\fx+.5,\dy*2);
\draw [line width=1.2pt,dashed,color=qqqqff] (\fx,\dy*2)-- (\fx-.5,\dy*2);
\draw [line width=1.2pt,dashed,color=qqqqff] (\fx,\dy*1.33)-- (\fx+.5,\dy*1.33);
\draw [line width=1.2pt,dashed,color=qqqqff] (\fx,\dy*1.33)-- (\fx-.5,\dy*1.33);
\draw [line width=1.2pt,dashed,color=qqqqff] (\fx,\dy*.67)-- (\fx+.5,\dy*.67);
\draw [line width=1.2pt,dashed,color=qqqqff] (\fx,\dy*.67)-- (\fx-.5,\dy*.67);
\draw [line width=1.2pt,dashed,color=qqqqff] (\rrrx,\rrry)-- (\rrrx,\rrry+.5);
\draw [line width=1.2pt,dashed,color=qqqqff] (\rrrx,\rrry)-- (\rrrx,\rrry-.5);
\draw [line width=1.2pt,dashed,color=qqqqff] (\rrrrx,\rrrry)-- (\rrrrx,\rrrry+.5);
\draw [line width=1.2pt,dashed,color=qqqqff] (\rrrrx,\rrrry)-- (\rrrrx,\rrrry-.5);
\draw [line width=1.2pt,dashed,color=qqqqff] (\rux,\ruy)-- (\rux+.5,\ruy);
\draw [line width=1.2pt,dashed,color=qqqqff] (\rux,\ruy)-- (\rux,\ruy+.5);
\draw [line width=1.2pt,dashed,color=qqqqff] (\rux,\ruy)-- (\rux-.5,\ruy);
\draw [line width=1.2pt,dashed,color=qqqqff] (\rdx,\rdy)-- (\rdx+.5,\rdy);
\draw [line width=1.2pt,dashed,color=qqqqff] (\rdx,\rdy)-- (\rdx,\rdy-.5);
\draw [line width=1.2pt,dashed,color=qqqqff] (\rdx,\rdy)-- (\rdx-.5,\rdy);
\draw [line width=1.2pt,dashed,color=qqqqff] (\rrux,\rruy)-- (\rrux+.5,\rruy);
\draw [line width=1.2pt,dashed,color=qqqqff] (\rrux,\rruy)-- (\rrux,\rruy+.5);
\draw [line width=1.2pt,dashed,color=qqqqff] (\rrux,\rruy)-- (\rrux-.5,\rruy);
\draw [line width=1.2pt,dashed,color=qqqqff] (\rrdx,\rrdy)-- (\rrdx+.5,\rrdy);
\draw [line width=1.2pt,dashed,color=qqqqff] (\rrdx,\rrdy)-- (\rrdx,\rrdy-.5);
\draw [line width=1.2pt,dashed,color=qqqqff] (\rrdx,\rrdy)-- (\rrdx-.5,\rrdy);

\draw [line width=1.2pt,dashed,color=qqqqff] (\rrrrrx,\rrrrry)-- (\rrrrrx+.5,\rrrrry);
\draw [line width=1.2pt,dashed,color=qqqqff] (\rrrrrx,\rrrrry)-- (\rrrrrx,\rrrrry-.5);
\draw [line width=1.2pt,dashed,color=qqqqff] (\rrrrrx,\rrrrry)-- (\rrrrrx,\rrrrry+.5);

\end{tikzpicture}
\caption{The tree of ELEPs; the elements $(a,b)$ are shown above the corresponding pairs $(s,t)$.}\label{F:tree}
\end{figure}


\section{Diagonals, heights and altitudes}\label{S:dah} 

Let us first fix some terminology and notation; see Figure \ref{F:terms}.

\begin{definition}
Consider a non-square ELEP $P$. We denote the length of its long (resp.~short) diagonal $d_l$ (resp.~$d_s$). The \emph{heights} of $P$ are the distances between opposite sides; we denote the long (resp.~short) height $h_l$ (resp.~$h_s$). Notice that the long (resp.~short) height connects 
short (resp.~long) sides.
Each diagonal $d$ partitions $P$ into two congruent triangles $T$. We will call the distance from $d$ to the third vertex of $T$ an  \emph{altitude} of $P$. We call the altitude from $d_s$ (resp.~$d_l$) the \emph{long} (resp.~\emph{short}) altitude and denote it $\eta_l$ (resp.~$\eta_s$). 
\end{definition}

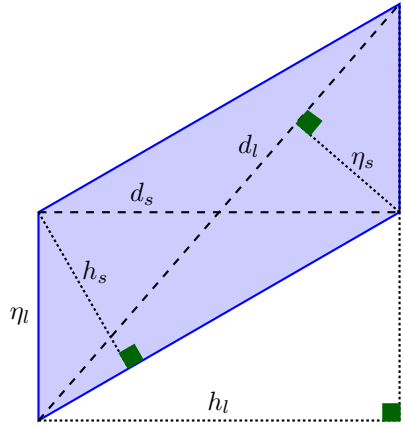
\begin{figure}[h]
\begin{tikzpicture}[scale=.8]
\def\r{1.732};

\draw [fill, blue!20] (0,0)
   -- (0, -4*\r/2)  
  -- (6,0) 
  -- (6,4*\r/2) 
 -- cycle;

\draw [thick, blue] (0,0)
   -- (0, -4*\r/2)  
  -- (6,0) 
  -- (6,4*\r/2) 
  -- cycle;

\draw [line width=0.8pt,dashed] (0,0)-- (6,0);
\draw [line width=0.8pt,dashed] (0,-4*\r/2)-- (6,4*\r/2);

\draw [line width=0.8pt,dash pattern=on 1pt off 1pt] (0,-4*\r/2)-- (6,-4*\r/2);
\draw [line width=0.8pt,dash pattern=on 1pt off 1pt] (6,-4*\r/2)-- (6,4*\r/2);
\draw [line width=0.8pt,dash pattern=on 1pt off 1pt] (0,0)-- (.85*\r,-3*.85);

\draw [line width=0.8pt,dash pattern=on 1pt off 1pt] (6,0)-- (6-.245*4*\r,.245*6);

\draw[line width=0.8pt,color=qqwuqq,fill=qqwuqq,fill opacity=0.1] (6,-4*\r/2) -- (6,-4*\r/2+.27) -- (6-.27,-4*\r/2+.27) -- (6-.27,-4*\r/2) -- cycle; 
\draw[line width=0.8pt,color=qqwuqq,fill=qqwuqq,fill opacity=0.1] (.86*\r,-3*.86) -- (.86*\r+.234,-3*.86+.135) -- (.86*\r+.1,-3*.86+.370) -- (.86*\r-.135,-3*.86+.234) -- cycle; 

\draw[line width=0.8pt,color=qqwuqq,fill=qqwuqq,fill opacity=0.1] (6-.246*4*\r,.246*6) -- (6-.215*4*\r,.215*6) -- (6-.246*4*\r+.4,.246*6+.04) -- (6-.246*4*\r+.177,.246*6+.204) -- cycle;

\draw (3,-4*\r/2+.3) node {$h_l$};
\draw (.43*\r+.2,-3*.43+.3) node {$h_s$};
\draw (1*\r,.3) node {$d_s$};
\draw (3.5,1.1) node {$d_l$};
\draw (-.3,-\r) node {$\eta_l$};
\draw (5.4,.85) node {$\eta_s$};

  \end{tikzpicture}
\caption{Diagonals, heights and altitudes}\label{F:terms}
\end{figure}

\begin{lemma}\label{L:widths}
Suppose an ELEP $P$ has sides $a\sqrt3,b\sqrt3$ with $a,b\in \N$ and $a\le b$. Then 
\begin{enumerate}
\item[\textup{(a)}] $\displaystyle h_l=\frac{2(a+b)}a,\qquad h_s=\frac{2(a+b)}b$,
\item[\textup{(b)}] $\displaystyle \eta_l=\frac{2\sqrt3(a+b)}{d_s},\qquad \eta_s=\frac{2\sqrt3(a+b)}{d_l}$.
\end{enumerate}
\end{lemma}

\begin{proof}
By equability, the area of $P$ is $2\sqrt3(a+b)$, but the area is obviously also $a\sqrt3h_l$ and $b\sqrt3h_s$. This gives (a).
But the area of $P$ is also twice the area of the triangle determined by each diagonal. So the area of $P$ is both $\eta_l d_s$ and $\eta_s d_l$.  This gives~(b).
\end{proof}

\begin{remark}\label{R:widths}
If  $h_s$  is an integer, then $h_s=2+\frac{2a}b$ by Lemma~\ref{L:widths}(a), and so as $b>a$, the only possibility is $b=2a$. But then
$9a^2b^2-12(a+b)^2=36 a^2 (a^2-3)$, which is a square only when $a=2$. This is the root parallelogram $(a,b)=(2,4)$. It has $h_s=3$ and  $h_l=6$.  Similarly,
if  $h_l$  is an integer, then $h_l=2+\frac{2b}a$ by Lemma~\ref{L:widths}(a), and so $2b=ak$, for some positive integer $k$. In this case, $h_l=2+k$ and $h_s=2+\frac{4}k$.  \end{remark}

\begin{theorem}\label{T:altbd} For every ELEP the altitudes satisfy  $2<\eta_s\le \frac6{\sqrt7}$ and $2<\eta_l\le 2\sqrt3$.
\end{theorem}

\begin{remark}\label{R:diags}
Suppose an ELEP $P$ has  sides $a\sqrt3,b\sqrt3$ with $a,b\in \N$ and $a\le b$. 
By Theorem~ \ref{T:345}, we have
$ab= 2(s^2+3t^2)$ and $a+b=6st$, so $\sqrt{9a^2 b^2 -12(a+b)^2}= 6|s^2-3t^2|$,
for some odd, coprime integers $s,t$.
If $s^2>3t^2$, then from Lemma \ref{L:diag}, the diagonals are given by
\begin{align*}
d_l^2&=3(a^2  +  b^2) +2\sqrt{9a^2 b^2 -12(a+b)^2} \\
&=3\cdot 36 s^2t^2-12(s^2+3t^2)+12|s^2-3t^2|  =12(9 s^2t^2-6t^2), \\
d_s^2&=3\cdot36 s^2t^2-12(s^2+3t^2)-12|s^2-3t^2| =12(9s^2t^2-2s^2).
\end{align*}
Similarly, if $s^2<3t^2$, the diagonals are given by
$d_l^2=12(9s^2t^2-2s^2)$ and $d_s^2=12(9s^2t^2- 6t^2)$.
Putting the cases together, we have
\begin{equation}\label{E:dlds}
d_l^2=12(9s^2t^2-2\min(s^2,3t^2)),\qquad d_s^2=12(9s^2t^2-2\max(s^2,3t^2)).
 \end{equation}
Notice that when a diagonal $d$ satisfies $d^2=12(9 s^2t^2-6t^2)$, one has $d=6t\sqrt{3s^2-2}$, so $d$ is an integer if and only if $3s^2-2$ is a square.
When $d^2=12(9 s^2t^2-2s^2)$, one has $d=2s\sqrt{27t^2-6}$. In this case, $d$ is never an integer, since $27t^2-6\equiv 3\pmod 9$ and $3$ is not a quadratic residue modulo 9.
\end{remark}

\begin{proof}
Suppose an ELEP $P$ has  sides $a\sqrt3,b\sqrt3$ with $a,b\in \N$ and $a\le b$. 
By Theorem~ \ref{T:345}, we have
$ab= 2(s^2+3t^2)$ and $a+b=6st$, so $\sqrt{9a^2 b^2 -12(a+b)^2}= 6|s^2-3t^2|$,
for some odd, coprime integers $s,t$.
Using Lemma~\ref{L:widths} and Remark~\ref{R:diags},
the altitudes are given by
\begin{align*}
\eta_l^2&=\frac{12(a+b)^2}{d_s^2}=\frac{36  s^2t^2}{9s^2t^2-2\max(s^2,3t^2)},\\
\eta_s^2&=\frac{12(a+b)^2}{d_l^2}=\frac{36  s^2t^2}{9s^2t^2-2\min(s^2,3t^2)}.
\end{align*}
We consider the two cases according to whether $s^2>3t^2$ or $s^2<3t^2$. First suppose $s^2>3t^2$.
Then
\begin{align*}
\eta_l^2&=\frac{36  s^2t^2}{9 s^2t^2-2s^2}=\frac{36  t^2}{9 t^2-2}=4+\frac{8}{9t^2-2 },\\
\eta_s^2&=\frac{36  s^2t^2}{9 s^2t^2-6t^2}=\frac{12  s^2}{3s^2-2 }=4+\frac{8}{3s^2-2 }.
\end{align*}
Hence, as $t\ge 1$, we have $4<\eta_l^2\le 4+\frac{8}{7}$, so $2<\eta_l\le \frac6{\sqrt7}$.
Similarly, as $s\ge 1$, we have $4<\eta_s^2\le 12$, so $2<\eta_s\le 2\sqrt3$. Thus, as $\eta_s<\eta_l$, we have $2<\eta_s<\eta_l\le \frac6{\sqrt7}$.

Now suppose $s^2<3t^2$.
Then
\begin{align*}
\eta_l^2&=\frac{36  s^2t^2}{9 s^2t^2-6t^2}=\frac{12  s^2}{3s^2-2 }=4+\frac{8}{3s^2-2 }\\
\eta_s^2&=\frac{36  s^2t^2}{9 s^2t^2-2s^2}=\frac{36  t^2}{9t^2-2 }=4+\frac{8}{9t^2-2 }.
\end{align*}
Hence, as 
$s\ge 1$, we have $4<\eta_l^2\le 12$, so $2<\eta_l\le 2\sqrt3$. 
Similarly, as 
$t\ge 1$, we have $4<\eta_s^2\le 4+\frac{8}{7}$, so $2<\eta_s\le \frac6{\sqrt7}$.
Thus, as $\eta_s<\eta_l$, we have $2<\eta_s<\eta_l\le 2\sqrt3$. 

Combining the two cases gives the required bounds on  $\eta_s$ and $\eta_l$. 
\end{proof}


\section{ELEPs with horizontal diagonal}\label{S:horh}

\begin{example}\label{E:hor}
Here we give examples of ELEPs with a horizontal diagonal.
Consider the sequence $(u_n)$ defined by $u_{n}= 4 u_{n-1} - u_{n-2}$, with $u_0=0$ and $u_1=1$. The following result is well known; see 
\cite[Chap.~5, Ex.~2.1]{Barb} and sequence A001353 in \cite{OEIS}. We provide a proof for completeness.

\begin{lemma}\label{L:xrec} For all $n\ge 1$, one has
\begin{enumerate}
\item[\textup{(a)}] $u_{n+1}u_{n-1}=u_n^2-1$,
\qquad  \textup{ (b)} $3u_{n}^2+1=(2u_n - u_{n-1})^2$.
\end{enumerate}
\end{lemma}

\begin{proof}(a) One has $u_{2}u_{0}=0=u_1^2-1$. Then for $n\ge 2$, by induction,
\begin{align*}
u_{n + 1} u_{n - 1} &- u_{n}^2
= (4 u_{n}  -u_{n - 1})u_{n - 1} - u_{n}(4 u_{n - 1}- u_{n - 2})\\
&=u_{n} u_{n - 2} - u_{n - 1}^2=-1.
\end{align*}

(b). For all $n\ge 1$, expanding $(2u_n - u_{n-1})^2$, one has, using (a),
\begin{align*}
3u_{n}^2+1-(2u_n - u_{n-1})^2&=1- u_{n}^2 +u_{n-1}(4u_n- u_{n-1})\\
&=1- u_{n}^2 + u_{n-1}u_{n+1}=0.\qedhere
\end{align*}
\end{proof}

Let 
\[
A_n=
(-6u_{n}, -2 \sqrt{3}),\quad
B_n=
(6u_{n+1} -6u_{n},0),\quad
C_n = 
(6u_{n+1},2 \sqrt{3}).
\]
The vertices $OA_nB_nC_n$ form a parallelogram on the Eisenstein lattice with diagonal $OB_n$ on the $x$-axis.  The parallelogram $OA_nB_nC_n$  has side lengths 
\begin{align*}
\overline{OA_n}&=\sqrt{36 u_{n}^2+12}=2\sqrt{3}\sqrt{3 u_{n}^2+1}=2\sqrt{3}(2u_{n} - u_{n-1}),\\
\overline{OC_n}&=
\sqrt{36 u_{n+1}^2+12}=2\sqrt{3}\sqrt{3 u_{n+1}^2+1}=2\sqrt{3}(2u_{n+1} - u_{n}),
\end{align*}
by Lemma~\ref{L:xrec}(b), and area
\[
\text{area}(OA_nB_nC_n)=\det\begin{pmatrix}
-6u_{n}& -2 \sqrt{3} \\ 
6u_{n+1}&2 \sqrt{3}
\end{pmatrix}=12 \sqrt{3}(u_{n+1}-u_{n}).
\]
Hence 
\begin{align*}
2(\overline{OA_n}+\overline{OC_n})&=4\sqrt{3}(2u_{n+1}+u_{n}- u_{n-1})\\
&=4\sqrt{3}(2u_{n+1} +u_{n}+ (u_{n+1}-4u_{n}))\\
&=12\sqrt{3}(u_{n+1} - u_{n})
=\text{area}(OA_nB_nC_n),
\end{align*}
so $OA_nB_nC_n$ is equable.

Table \ref{T:hd} lists the first 9 examples of ELEPs with horizontal diagonal. 
Note that the first four of these ELEPs appear in Figure~\ref{F:tree} in the branch that starts at the root $(2,4)$ and descends vertically.
Figure~\ref{F:hd} shows the first two examples, with the first one translated 6 units to the left. 

\begin{table}[h]
\begin{center}
\begin{tabular}{c|cccccc}
  \hline
  $n$ & $(q,u)$ & $a$ & $b$  & $A$& $B$  & $(s,t)$ \\
  \hline
  $0$ & $(1,0)$ & $2$ & $4$ & $-2-4\omega$ & $6$ & $(1,1)$\\
$1$ & $(2,1)$ & $4$ &  $14$ & $-8-4\omega$ & $18$ & $(1,3)$\\
$2$ & $(7,4)$ & $14$ & $52$ & $-26-4\omega$ & $66$ & $(1,11)$\\
$3$ & $(26,15)$ & $52$ & $194$ & $-92-4\omega$ & $246$ & $(1,41)$\\
$4$ & $(97,56)$ & $194$ & $724$ & $-338-4\omega$ & $918$ & $(1,153)$\\
$5$ & $(362,209)$ & $724$ & $2702$ & $-1256-4\omega$ & $3426$ & $(1,571)$\\
$6$ & $(1351,780)$ & $2702$ & $10084$ & $-4682-4\omega$ & $12786$ & $(1,2131)$\\
$7$ & $(5042,2911)$ & $10084$ & $37634$ & $-17468-4\omega$ & $47718$ & $(1,7953)$\\
$8$ & $(18817,10864)$ & $37634$ & $140452$ & $-65186-4\omega$ & $178086$ & $(1,29681)$\\
 \hline
\end{tabular}
\end{center}
\caption{Examples of ELEPs with horizontal diagonal}
\label{T:hd}
\end{table}

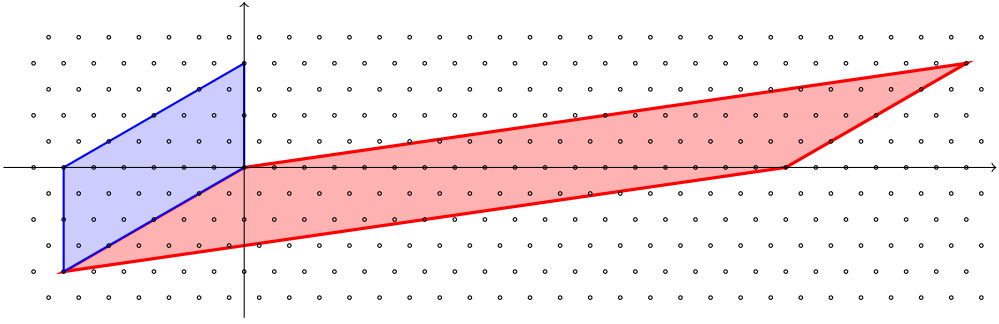
\begin{figure}[h]
\begin{tikzpicture}[scale=.4]
\def\r{1.732};

\draw [fill, red!30] (0,0) 
  -- (-8+4/2, -4*\r/2)  
  -- (18,0) 
  -- (24,4*\r/2) 
  -- cycle;

\draw [very thick, red] (0,0) 
   -- (-8+4/2, -4*\r/2)  
  -- (18,0) 
  -- (24,4*\r/2) 
 -- cycle;

\draw [fill, blue!20] (-6,0)
   -- (-6, -4*\r/2)  
  -- (0,0) 
  -- (0,4*\r/2) 
 -- cycle;

\draw [thick, blue] (-6,0)
   -- (-6, -4*\r/2)  
  -- (0,0) 
  -- (0,4*\r/2) 
  -- cycle;

\foreach \i in {-7,...,24}
\foreach \j in {-2,...,2}
\draw (\i,\j*\r) circle (.06);
\foreach \i in {-6,...,25}
\foreach \j in {-3,...,2}
\draw (\i-1/2,\r*\j+\r/2) circle (.06);

 \draw [->] (0,-5) -- (0,5.5);
 \draw [->] (-8,0) -- (25,0);

  \end{tikzpicture}
\caption{Two ELEPs with horizontal diagonal}\label{F:hd}
\end{figure}
\end{example}

\begin{theorem}\label{T:hor}
Up to Euclidean isometry, the only ELEPs with horizontal diagonal are those of Example~\ref{E:hor}.
\end{theorem}

\begin{proof} By translating and reflecting in the $x$ and/or $y$ axes if necessary, we may assume that the horizontal diagonal lies on the positive $x$-axis, starting at the origin $0$, and that the side starting at $0$, and lying in the 3rd or 4th quadrants, is the shorter of the two sides. Therefore, suppose we have a ELEP with vertices $A=x-y\omega, B=z,C=(z-x)+y\omega, O=0$, where $x\in\Z$ and $y,z\in \N$. Let $a\sqrt3,b\sqrt3$ denote the lengths of $OA$ and $AB$ respectively, with $a,b\in\N$ and $a<b$. In particular, we have
\begin{equation}\label{E:OA}
3a^2=x^2+xy+y^2.
\end{equation}
The altitudes from $A$ and $C$ have length $\eta:=\frac{y\sqrt3}2$. By Theorem~\ref{T:altbd},  $2<\eta\le 2\sqrt3$, which gives $\frac{4}{\sqrt3}<y\le 4$. So, as $y\in\N$, we have $y=3$ or $4$. First suppose that $y=3$. Then \eqref{E:OA} gives
\begin{equation}\label{E:a}
3a^2=x^2+3x+9.
\end{equation}
 So $x$ is divisible by $3$, say $x=3x'$. But then by \eqref{E:a}, 
 $a$ is divisible by $3$, contrary to Theorem~\ref{T:345}. So we have $y=4$.

As $y=4$, we have $\eta=2\sqrt3 > \frac6{\sqrt7}$. So $\eta=\eta_l$ by Theorem~\ref{T:altbd}, and from the proof of Theorem~\ref{T:altbd} we have  $s^2<3t^2$ and
$\eta_l^2=4+\frac{8}{3s^2-2 }$. Hence $(2\sqrt3)^2= 4+\frac{8}{3s^2-2 }$, giving $s=1$.
Then by Theorem~\ref{T:345},
$(a,b)=2(q, r) \in\N^2$ where 
\begin{equation}\label{E:sdsl}
1 + 3t^2 + 2q^2 = 6tq
 \end{equation}
 and $q \le r =3t-q$. Let $u:=t-q$. From \eqref{E:sdsl}, we have
\begin{equation}\label{E:pell}
 q^2-3u^2=1. 
 \end{equation}
 Notice that $u\ge 0$. Indeed, from \eqref{E:sdsl}, we have
 $q=\frac{3t-\sqrt{3t^2-2}}2$, so $u=\frac{-t+\sqrt{3t^2-2}}2\ge 0$, for $t\ge 1$.
 Notice also that $a^2=4q^2$ and \eqref{E:OA} gives $3a^2=x^2+4x+16$, so $12q^2=x^2+4x+16$. Thus $x$ is even, say $x=2x'$, and so $3q^2=x'^2+2x'+4$. Hence, by \eqref{E:pell}, $9u^2=x'^2+2x'+1=(x'+1)^2$, so
\begin{equation}\label{E:2c}
x=2x'=\pm 6u-2.
\end{equation}

Now \eqref{E:pell} is one of Pell's equations and it is well known (see \cite{De}) that the solutions $(q_n,u_n)$ to \eqref{E:pell} are given by the recurrence relation
 \[
 (q_{n},u_n)=4(q_{n-1},u_{n-1})-(q_{n-2},u_{n-2}),\qquad (q_{1},u_1)=(2,1),\ (q_{0},u_0)=(1,0).
 \]
In particular, $u_n$ agrees with the sequence $u_n$ of Example~\ref{E:hor}. Notice also that 
\begin{equation}\label{E:qu}
q_n+2u_n = u_{n+1}, \quad\text{for all } n\ge 0.
 \end{equation}
Indeed, both sides of the equation satisfy the same recurrence relation, with the initial conditions $q_0+2u_0 = 1= u_1$ and $q_1+2u_1 =4= u_2$.

For each $n\ge 0$, we denote by $a_n,b_n,x_n,z_n,t_n$ the values of $a,b,x,z,t$, respectively, determined by $q_n$ and $u_n$, and we denote the corresponding parallelogram  $OA_nB_nC_n$. Since there are so many variables, we recall for the reader's convenience that 
$A_n=x_n-4\omega, B_n=z_n, C=(z_n-x_n)+4\omega$, the lengths of $OA_n, OC_n$ are $\sqrt3a_n,\sqrt3b_n$ respectively, and from the definition of $u$, we also have $a_n+b_n=6t_n$ and $t_n=u_n+q_n$. Furthermore, $a_n=2q_n$ and 
\begin{equation}\label{E:bn}
b_n=6t_n-a_n=6t_n-2q_n=6u_n+4q_n.
\end{equation}
Since $A_n$ has $y$-coordinate $-2\sqrt3$, the area of $OA_nB_nC_n$ is $2\sqrt3 z_n$, so the equability condition is $2\sqrt3 z_n=2(a_n+b_n)\sqrt3$; that is, $z_n=a_n+b_n=6t_n$. 
So $z_n=6(u_n+q_n)$. 

We now consider the two cases given by \eqref{E:2c}: (i) $x_n=6u_n-2$ for $u_n>0$ and (ii) $x_n=- 6u_n-2$ for $u_n\ge0$.

(i). We show that this case leads to a contradiction. Suppose $x_n=6u_n-2$  for $u_n>0$, so $A_n=6u_n-2-4\omega, B_n=z_n=6(u_n+q_n)$ and $C_n= 6q_n+2+4\omega$. So, from \eqref{E:bn},
\[
3\cdot (6u_n+4q_n)^2=3b_n^2=\overline{OC_n}^2=(6q_n+2)^2-(6q_n+2)4+16,
\]
where the expression on the right comes for the square length of the element $C_n=6q_n+2+4\omega$. Expanding, rearranging and dividing by 12 gives $q_n^2 + 12 q_n u_n + 9 u_n^2=1$. But this is impossible for positive integers $u_n$ and $q_n$. 

(ii). Now suppose $x_n=- 6u_n-2$ for $u_n\ge0$. Using \eqref{E:qu}, we have $A_n=-6u_n-2-4\omega, B_n=z_n=6(u_n+q_n)=6u_{n+1} -6u_{n}$,
and $C_n=6u_{n+1} + 2  + 4\omega$, which is one of the parallelograms of Example~\ref{E:hor}. 
\end{proof}

Since the parallelograms of Example~\ref{E:hor} have horizontal short diagonal, we have the following conclusion.

 \begin{corollary} No ELEP has a horizontal long diagonal. 
\end{corollary}

\begin{remark}
The upper bounds $\eta_s= \frac6{\sqrt7}$  and $\eta_l= 2\sqrt3$ of Theorem~\ref{T:altbd} are both attained by the first of the parallelograms of Example~\ref{E:hor}; that is, the root parallelogram $(a,b)=(2,4)$ given by $n=0$ in Table~\ref{T:hor}.
\end{remark}


\section{ELEPs with a vertical diagonal}\label{S:vertd}

\begin{theorem}\label{T:verd}
There is no ELEP having a vertical diagonal.
\end{theorem}

\begin{proof} Suppose we have an ELEP with a vertical diagonal.
By translating and reflecting in the $x$ and/or $y$ axes if necessary, we may assume that the vertical diagonal lies on the positive $y$-axis, starting at the origin $0$, and that the side starting at $0$, and lying in the 1st or 2nd quadrants, is the shorter of the two sides. Therefore, suppose we have an ELEP with vertices $A=x+y\omega, B=z\omega,C=-x+(z-y)\omega, O=0$, where $x,z\in\N$ and $y\in \Z$. Let $a\sqrt3,b\sqrt3$ denote the lengths of $OA$ and $AB$ respectively, with $a,b\in\N$ and $a<b$. In particular, we have
\begin{equation}\label{E:OAv}
3a^2=x^2-xy+y^2.
\end{equation}
The altitudes from $A$ and $C$ have length $\eta:=x$. By Theorem~\ref{T:altbd},  $2<\eta\le 2\sqrt3$, which gives $x=3$. 
Then \eqref{E:OAv} gives
\begin{equation}\label{E:av}
3a^2=9-3y+y^2.
\end{equation}
 So $y$ is divisible by $3$, say $y=3y'$. But then by \eqref{E:av}, 
 $a$ is divisible by $3$, contrary to Theorem~\ref{T:345}. 
\end{proof}

\section{ELEPs with a vertical side}\label{S:verts} 

\begin{example}\label{E:verts2}
Consider the Pell-like equation 
\begin{equation}\label{E:Pell}
3u^2=2m^2+1.
\end{equation}
This equation is well known; see entry A072256 in \cite{OEIS}.
Its solutions $(u_n,m_n)$ satisfy the recurrence relation 
\[
(u_n,m_n) = 10(u_{n-1},m_{n-1})  - (u_{n-2},m_{n-2}),
\]
with $(u_1,m_{1}) = (1,1), (u_2,m_{2})=(9,11)$.

Now, using complex numbers, consider the parallelogram with vertices 
$A_n=2\sqrt3i,B_n=-x_n+(2+y_n)\sqrt3i,C_n=-x_n+y_n\sqrt3i, O=0$, where  
$x_n=3+3u_n$ and $y_n=2m_n$. Note that the vertical side $OA_n$ has length $a\sqrt3$, where $a:=2$, while the side $OC_n$ has length 
\begin{align*}
b_n\sqrt3=\sqrt{x_n^2+3y_n^2}&=\sqrt{(3+3u_n)^2+12m_n^2}\\&=\sqrt{9+18u_n+ 9u_n^2+6(3u_n^2-1)}\quad(\text{by }\eqref{E:Pell})\\
&=\sqrt3(1+3u_n).
\end{align*}
 So $OA_nB_nC_n$ has perimeter
$2(a+b_n)\sqrt3
=6\sqrt3(1+u_n)$.
Furthermore, $OA_nB_nC_n$ has area $a\sqrt3x_n=2x_n\sqrt3=6\sqrt3(1+u_n)$. So $OA_nB_nC_n$ is equable.

Table \ref{T:vs} lists the first 9 of these examples. 
The first five of these ELEPs  appear in Figure~\ref{F:a4}, and in Figure~\ref{F:tree}, in the branch that starts at the root $(2,4)$ and proceeds horizontally to the right. The values of $(s,t)$ were computed using \eqref{E:phimatrices}.
Figure~\ref{F:vs} shows the first two examples, with the first one reflected in the $y$-axis.

\begin{table}[h]
\begin{center}
\begin{tabular}{c|cccccc}
  \hline
  $n$ & $(u,m)$ &  $b$  & $x$ & $y$  & $(s,t)$ \\
  \hline
  $1$ & $(1,1)$ &  $4$ & $6$ &$2$ &  $(1,1)$\\
$2$ & $(9,11)$ &   $28$ & $30$ &$22$ & $(5,1)$\\
$3$ & $(89,109)$ &  $268$ & $270$ &$218$ &  $(5,9)$\\
$4$ & $(881,1079)$ &  $2644$ & $2646$ &$2158$ &  $(49,9)$\\
$5$ & $(8721,10681)$ &  $26164$ & $26166$ &$21362$ &  $(49,89)$\\
$6$ & $(86329,105731)$ &  $258988$ & $258990$ &$211462$ &  $(485,89)$\\
$7$ & $(854569,1046629)$ &  $2563708$ & $2563710$ &$2093258$ &  $(485, 881)$\\
$8$ & $(8459361,10360559)$ &  $25378084$ & $25378086$ &$20721118$ &  $(4801, 881)$\\
$9$ & $(83739041,102558961)$ &  $2486793150$ & $2030458102$ & $205117922$ & $(4801, 8721)$\\
 \hline
\end{tabular}
\end{center}
\caption{Examples of ELEPs with vertical side of length $2\sqrt3$ and vertex $C=-x+y\sqrt3i$}
\label{T:vs}
\end{table}

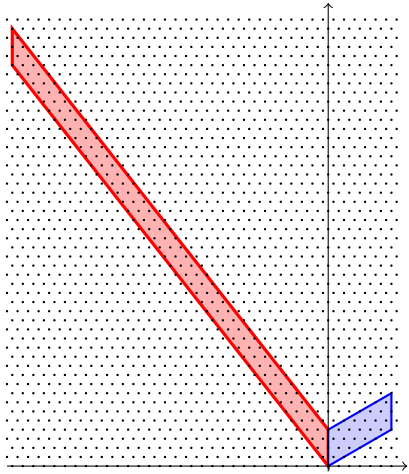
\begin{figure}[h]
\begin{tikzpicture}[scale=.14]
\def\r{1.732};

\draw [fill, red!30] (0,0) 
     -- (0,4*\r/2) 
  -- (-30, 48*\r/2)  
 -- (-30, 44*\r/2)  
  -- cycle;

\draw [very thick, red] (0,0) 
    -- (0,4*\r/2) 
  -- (-30, 48*\r/2)  
 -- (-30, 44*\r/2)  
 -- cycle;

\draw [fill, blue!20] (0,0)
   -- (0,4*\r/2) 
  -- (6, 8*\r/2)  
 -- (6, 4*\r/2)  
   -- cycle;

\draw [thick, blue] (0,0)
   -- (0,4*\r/2) 
  -- (6, 8*\r/2)  
 -- (6, 4*\r/2)  
   -- cycle;

\foreach \i in {-30,...,6}
\foreach \j in {0,...,24}
\draw (\i,\j*\r) circle (.06);
\foreach \i in {-30,...,7}
\foreach \j in {0,...,24}
\draw (\i-1/2,\r*\j+\r/2) circle (.06);

 \draw [->] (-30.5,0) -- (7.5,0);
 \draw [->] (0,-.5) -- (0,44);

  \end{tikzpicture}
\caption{Two ELEPs having a vertical side of length $2\sqrt3$.}\label{F:vs}
\end{figure}
\end{example}

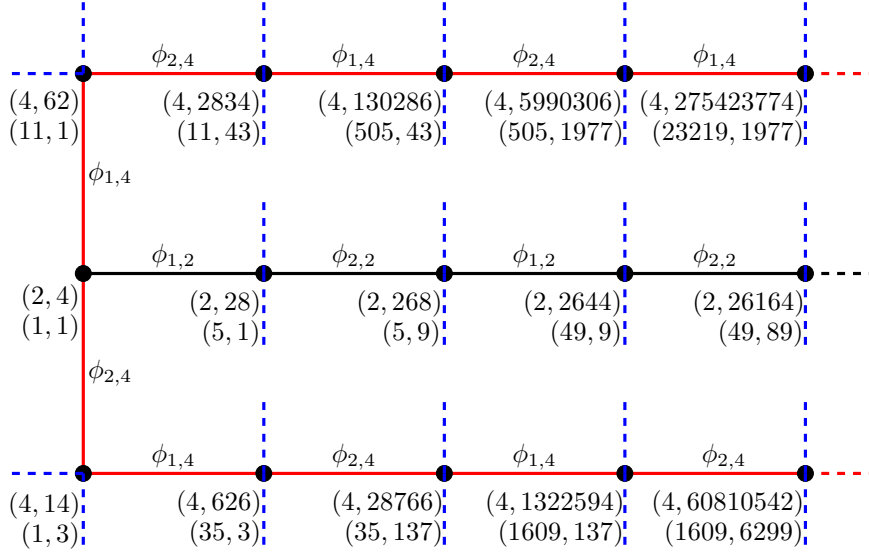
\begin{figure}[h]
\begin{tikzpicture}[scale=2][line cap=round,line join=round,>=triangle 45,x=1cm,y=1cm]
\def\fx{0};
\def\fy{0};
\def\rx{1.2};
\def\ry{0};
\def\rrx{2.4};
\def\rry{0};
\def\rrrx{3.6};
\def\rrry{0};
\def\rrrrx{4.8};
\def\rrrry{0};
\def\rrrrrx{5};
\def\rrrrry{0};
\def\rux{1};
\def\ruy{1};
\def\rdx{1};
\def\rdy{-1};
\def\rrux{2};
\def\rruy{2};
\def\rrdx{2};
\def\rrdy{-2};
\def\ux{0};
\def\uy{1};
\def\uuy{2};
\def\uuuy{3};
\def\dx{0};
\def\dy{-1};
\def\ddy{-2};
\def\dddy{-3};
\draw [line width=1.2pt,color=red] (\fx,\fy)-- (\ux,\uy*1.33-.05);
\draw [line width=1.2pt,color=red] (\fx,\fy)-- (\dx,\dy*1.33+.05);
\draw [line width=1.2pt,color=red] (\fx,\uy*1.33)-- (\rrrrx,\uy*1.33);
\draw [line width=1.2pt,color=red] (\fx,\dy*1.33)-- (\rrrrx,\dy*1.33);
\draw [line width=1.2pt,dashed,color=red] (\rrrrx,\uy*1.33)-- (\rrrrx+.5,\uy*1.33);
\draw [line width=1.2pt,dashed,color=red] (\rrrrx,\dy*1.33)-- (\rrrrx+.5,\dy*1.33);

\draw [line width=1.2pt,color=black] (\fx,\fy)--  (\rrrrx,\rrrry);
\draw [line width=1.2pt,dashed,color=black] (\rrrrx,\rrrry)-- (\rrrrx+.5,\rrrry);

\draw [fill=black] (\fx,\fy) circle (1.5pt);
\draw[color=black] (\fx-.22,\fy-.16)  node {$(2,4)$};
\draw[color=black] (\fx-.22,\fy-.36)  node {$(1,1)$};
\draw[color=black] (\fx+.6,\fy+.11)  node {$\phi_{1,2}$};
\draw [fill=black] (\rx,\ry) circle (1.5pt);
\draw[color=black] (\rx-.26,\ry-.2) node {$(2,28)$};
\draw[color=black] (\rx-.22,\ry-.4) node {$(5,1)$};
\draw[color=black] (\rx+.6,\ry+.11)  node {$\phi_{2,2}$};
\draw [fill=black] (\rrx,\rry) circle (1.5pt);
\draw[color=black] (\rrx-.3,\rry-.2) node {$(2,268)$};
\draw[color=black] (\rrx-.22,\rry-.4) node {$(5,9)$};
\draw[color=black] (\rrx+.6,\rry+.11)  node {$\phi_{1,2}$};
\draw [fill=black] (\rrrx,\rrry) circle (1.5pt);
\draw[color=black] (\rrrx-.35,\rrry-.2) node {$(2,2644)$};
\draw[color=black] (\rrrx-.26,\rrry-.4) node {$(49,9)$};
\draw[color=black] (\rrrx+.6,\rrry+.11)  node {$\phi_{2,2}$};
\draw [fill=black] (\rrrrx,\rrrry) circle (1.5pt);
\draw[color=black] (\rrrrx-.38,\rrrry-.2) node {$(2,26164)$};
\draw[color=black] (\rrrrx-.3,\rrrry-.4) node {$(49,89)$};
\draw[color=black] (\fx-.26,\uuy*.67-.2) node {$(4,62)$};
\draw[color=black] (\fx-.26,\uuy*.67-.4) node {$(11,1)$};
\draw[color=black] (\fx+.18,\uuy*.33) node {$\phi_{1,4}$};
\draw[color=black] (\fx-.26,\ddy*.67-.2) node {$(4,14)$};
\draw[color=black] (\fx-.22,\ddy*.67-.4) node {$(1,3)$};
\draw[color=black] (\fx+.18,\ddy*.33) node {$\phi_{2,4}$};

\draw[color=black] (\fx+.6,\dy*1.33+.11)  node {$\phi_{1,4}$};
\draw[color=black] (\rx-.3,\dy*1.33-.2) node {$(4,626)$};
\draw[color=black] (\rx-.26,\dy*1.33-.4) node {$(35,3)$};
\draw[color=black] (\rx+.6,\dy*1.33+.11)  node {$\phi_{2,4}$};
\draw[color=black] (\rrx-.38,\dy*1.33-.2) node {$(4,28766)$};
\draw[color=black] (\rrx-.34,\dy*1.33-.4) node {$(35,137)$};
\draw[color=black] (\rrx+.6,\dy*1.33+.11)  node {$\phi_{1,4}$};
\draw[color=black] (\rrrx-.47,\dy*1.33-.2) node {$(4,1322594)$};
\draw[color=black] (\rrrx-.43,\dy*1.33-.4) node {$(1609,137)$};
\draw[color=black] (\rrrx+.66,\dy*1.33+.11)  node {$\phi_{2,4}$};
\draw[color=black] (\rrrrx-.52,\dy*1.33-.2) node {$(4,60810542)$};
\draw[color=black] (\rrrrx-.48,\dy*1.33-.4) node {$(1609,6299)$};

\draw[color=black] (\fx+.6,\uy*1.33+.11)  node {$\phi_{2,4}$};
\draw[color=black] (\rx-.35,\uy*1.33-.2) node {$(4,2834)$};
\draw[color=black] (\rx-.3,\uy*1.33-.4) node {$(11,43)$};
\draw[color=black] (\rx+.6,\uy*1.33+.11)  node {$\phi_{1,4}$};
\draw[color=black] (\rrx-.43,\uy*1.33-.2) node {$(4,130286)$};
\draw[color=black] (\rrx-.34,\uy*1.33-.4) node {$(505, 43)$};
\draw[color=black] (\rrx+.6,\uy*1.33+.11)  node {$\phi_{2,4}$};
\draw[color=black] (\rrrx-.48,\uy*1.33-.2) node {$(4,5990306)$};
\draw[color=black] (\rrrx-.43,\uy*1.33-.4) node {$(505, 1977)$};
\draw[color=black] (\rrrx+.6,\uy*1.33+.11)  node {$\phi_{1,4}$};
\draw[color=black] (\rrrrx-.56,\uy*1.33-.2) node {$(4,275423774)$};
\draw[color=black] (\rrrrx-.52,\uy*1.33-.4) node {$(23219, 1977)$};

\draw [fill=black] (\fx,\uy*1.33) circle (1.5pt);
\draw [fill=black] (\rx,\uy*1.33) circle (1.5pt);
\draw [fill=black] (\rrx,\uy*1.33) circle (1.5pt);
\draw [fill=black] (\rrrx,\uy*1.33) circle (1.5pt);
\draw [fill=black] (\rrrrx,\uy*1.33) circle (1.5pt);
\draw [fill=black] (\fx,\dy*1.33) circle (1.5pt);
\draw [fill=black] (\rx,\dy*1.33) circle (1.5pt);
\draw [fill=black] (\rrx,\dy*1.33) circle (1.5pt);
\draw [fill=black] (\rrrx,\dy*1.33) circle (1.5pt);
\draw [fill=black] (\rrrrx,\dy*1.33) circle (1.5pt);
\draw [line width=1.2pt,dashed,color=qqqqff] (\fx,\uy*1.33)-- (\fx,\uy*1.33+.5);
\draw [line width=1.2pt,dashed,color=qqqqff] (\fx,\uy*1.33)-- (\fx-.5,\uy*1.33);
\draw [line width=1.2pt,dashed,color=qqqqff] (\fx,\dy*1.33)-- (\fx,\dy*1.33-.5);
\draw [line width=1.2pt,dashed,color=qqqqff] (\fx,\dy*1.33)-- (\fx-.5,\dy*1.33);
\draw [line width=1.2pt,dashed,color=qqqqff] (\rrrx,\rrry)-- (\rrrx,\rrry+.5);
\draw [line width=1.2pt,dashed,color=qqqqff] (\rrrx,\rrry)-- (\rrrx,\rrry-.5);
\draw [line width=1.2pt,dashed,color=qqqqff] (\rrrrx,\rrrry)-- (\rrrrx,\rrrry+.5);
\draw [line width=1.2pt,dashed,color=qqqqff] (\rrrrx,\rrrry)-- (\rrrrx,\rrrry-.5);
\draw [line width=1.2pt,dashed,color=qqqqff] (\rx,\ry)--  (\rx,\ry+.5);
\draw [line width=1.2pt,dashed,color=qqqqff] (\rx,\ry)-- (\rx,\ry-.5);
\draw [line width=1.2pt,dashed,color=qqqqff] (\rrx,\rry)-- (\rrx,\rry+.5);
\draw [line width=1.2pt,dashed,color=qqqqff] (\rrx,\rry)-- (\rrx,\rry-.5);

\draw [line width=1.2pt,dashed,color=qqqqff] (\rx,\uy*1.33)-- (\rx,\uy*1.33+.5);
\draw [line width=1.2pt,dashed,color=qqqqff] (\rrx,\uy*1.33)-- (\rrx,\uy*1.33+.5);
\draw [line width=1.2pt,dashed,color=qqqqff] (\rrrx,\uy*1.33)-- (\rrrx,\uy*1.33+.5);
\draw [line width=1.2pt,dashed,color=qqqqff] (\rrrrx,\uy*1.33)-- (\rrrrx,\uy*1.33+.5);

\draw [line width=1.2pt,dashed,color=qqqqff] (\rx,\uy*1.33)-- (\rx,\uy*1.33-.5);
\draw [line width=1.2pt,dashed,color=qqqqff] (\rrx,\uy*1.33)-- (\rrx,\uy*1.33-.5);
\draw [line width=1.2pt,dashed,color=qqqqff] (\rrrx,\uy*1.33)-- (\rrrx,\uy*1.33-.5);
\draw [line width=1.2pt,dashed,color=qqqqff] (\rrrrx,\uy*1.33)-- (\rrrrx,\uy*1.33-.5);

\draw [line width=1.2pt,dashed,color=qqqqff] (\rx,\dy*1.33)-- (\rx,\dy*1.33+.5);
\draw [line width=1.2pt,dashed,color=qqqqff] (\rrx,\dy*1.33)-- (\rrx,\dy*1.33+.5);
\draw [line width=1.2pt,dashed,color=qqqqff] (\rrrx,\dy*1.33)-- (\rrrx,\dy*1.33+.5);
\draw [line width=1.2pt,dashed,color=qqqqff] (\rrrrx,\dy*1.33)-- (\rrrrx,\dy*1.33+.5);

\draw [line width=1.2pt,dashed,color=qqqqff] (\rx,\dy*1.33)-- (\rx,\dy*1.33-.5);
\draw [line width=1.2pt,dashed,color=qqqqff] (\rrx,\dy*1.33)-- (\rrx,\dy*1.33-.5);
\draw [line width=1.2pt,dashed,color=qqqqff] (\rrrx,\dy*1.33)-- (\rrrx,\dy*1.33-.5);
\draw [line width=1.2pt,dashed,color=qqqqff] (\rrrrx,\dy*1.33)-- (\rrrrx,\dy*1.33-.5);


\end{tikzpicture}
\caption{The paths of ELEPs  having a side of length $2$ or $4$.}\label{F:a4}
\end{figure}

\begin{example}\label{E:verts4}
We now give two sequences of ELEPs having a vertical side of length $4\sqrt3$.
Consider the equation
\begin{equation}\label{E:notPell}
132 w^2+ 36 w+1=y^2.
\end{equation}
This equation is not particularly well known. The first six solutions for $(w,y)$ are:
\[
\{1,13\},\ \{5,59\},\ \{52,599\},\ \{236,2713\},\ \{2397,27541\},\ \{10857,124739\}. 
\]
In fact, using Alpern's integer equation solver \cite{Alpern}, one finds that the sequence of solutions is composed of two interspersed sequences, having the same recurrence relation 
\[
\begin{pmatrix}w_n\\y_n\end{pmatrix}= \begin{pmatrix}23&2\\264&23\end{pmatrix}\begin{pmatrix}w_{n-1}\\y_{n-1}\end{pmatrix}+
\begin{pmatrix}3\\36\end{pmatrix},
\]
but different initial conditions. One sequence  has $(w_1,y_{1}) = (1,13)$, while the other has $(w_1,y_{1}) = (5,59)$.

Now, for each of the two sequences,  consider the parallelogram with vertices 
$A_n=4\sqrt3i,B_n=-x_n+(4+y_n)\sqrt3i,C_n=-x_n+y_n\sqrt3i, O=0$, where  
$x_n=6w_n+3$. Note that the vertical side $OA_n$ has length $a:=4\sqrt3$, while the side $OC_n$ has length 
\begin{align*}
b_n\sqrt3=\sqrt{x_n^2+3y_n^2}&=\sqrt{(6w_n+3)^2+3y_n^2}\\&=\sqrt{(6w_n+3)^2+3(132 w_n^2+ 36 w_n+1)}\quad(\text{by }\eqref{E:notPell})\\
&=2\sqrt3(1+6w_n).
\end{align*}
 So $OA_nB_nC_n$ has perimeter
$2(a+b_n)
=12\sqrt3(1+2w_n)$.
Furthermore, $OA_nB_nC_n$ has area $ax_n=4x_n\sqrt3=12\sqrt3(1+2w_n)$. So $OA_nB_nC_n$ is equable.

Table \ref{T:vs1} lists the first six examples of ELEPs corresponding to the  sequence with initial condition $(w_1,y_{1}) = (1,13)$. 
The first five of these ELEPs  appear in Figure~\ref{F:a4} in the branch that starts at the element $(a,b)=(4,14)$ and proceeds horizontally to the right.
Table \ref{T:vs2} lists the first six examples of ELEPs corresponding to the  sequence with initial condition $(w_1,y_{1}) = (5,59)$. 
The first five of these ELEPs  appear in Figure~\ref{F:a4} in the branch that starts at the element $(a,b)=(4,62)$ and proceeds horizontally to the right. Notice that the two sequences are actually the two ends of a bi-infinite path, and are connected by the vertical path of length $2$ that passes from $(4,14)$ through the root $(2,4)$ to $(4,62)$.

\begin{table}[h]
\begin{center}
\begin{tabular}{c|ccccc}
  \hline
  $n$ & $(w,y)$ &  $b$  & $x$   & $(s,t)$ \\
  \hline
  $1$ & $(1,13)$ &  $14$ & $9$ &  $(1,3)$\\
$2$ & $(52,599)$ & $626$ & $315$ & $(35,3)$ \\ 
$3$ & $(2397,27541)$ & $28766$ & $14385$ & $(35,137)$ \\ 
$4$ & $(110216,1266287)$ & $1322594$ & $661299$ & $(1609,137)$ \\ 
$5$ & $(5067545,58221661)$ & $60810542$ & $30405273$ & $(1609,6299)$ \\ 
$6$ & $(232996860,2676930119)$ & $2795962322$ & $1397981163$ & $(73979,6299)$ \\ 
\hline
\end{tabular}
\end{center}
\caption{First sequence of ELEPs with vertical side of length $4\sqrt3$ and vertex $C=-x+y\sqrt3i$}
\label{T:vs1}
\end{table}

\begin{table}[h]
\begin{center}
\begin{tabular}{c|cccccc}
  \hline
  $n$ & $(w,y)$ &  $b$  & $x$   & $(s,t)$ \\
  \hline
  $1$ & $(5,59)$ &  $62$ &$18$ &  $(11,1)$\\
$2$ & $(236,2713)$ & $2834$ & $1419$ & $(11,43)$ \\ 
$3$ & $(10857,124739)$ & $130286$ & $65145$ & $(505,43)$ \\ 
$4$ & $(499192,5735281)$ & $5990306$ & $2995155$ & $(505,1977)$ \\ 
$5$ & $(22951981,263698187)$ & $275423774$ & $137711889$ & $(23219,1977)$ \\ 
$6$ & $(1055291940,12124381321)$ & $12663503282$ & $6331751643$ & $(23219,90899)$ \\ 
\hline
\end{tabular}
\end{center}
\caption{Second sequence of ELEPs with vertical side of length $4\sqrt3$ and vertex $C=-x+y\sqrt3i$}
\label{T:vs2}
\end{table}

\end{example}

\begin{theorem}\label{T:vers}
Up to Euclidean transformations, the only ELEPs having a side of length $2\sqrt3$ or $4\sqrt3$ are those of Example~\ref{E:verts2} and Example~\ref{E:verts4}, respectively.
Moreover, all ELEPs with a vertical side have a side of length $2\sqrt3$ or $4\sqrt3$.
\end{theorem}

\begin{proof}  Let  $OABC$ be an ELEP with  sides of length $a\sqrt3$ and $b\sqrt3$. First suppose that $a=2$.
Then Lemma~\ref{L:diag} gives the integer
\[
\sqrt{9a^2b^2-12(a+b)^2}=2 \sqrt{6((b-1)^2-3)}\in\N.
\]
So $(b-1)^2-3=6m^2$ for some $m\in\N$. In particular, $b\equiv 1\pmod 3$, say $b=3u+1$, as in Example~\ref{E:verts2}. Thus
$3u^2=2m^2+1$, which is the same equation as \eqref{E:Pell}. Consequently, $b$ is one of the numbers $b_n$ of Example~\ref{E:verts2}. It follows that, up to a Euclidean transformation, the $OABC$ is one of the ELEPs of Example~\ref{E:verts2}.

Now suppose that $a=4$.
Then Lemma~\ref{L:diag} gives the integer
\[
\sqrt{9a^2b^2-12(a+b)^2}=2 \sqrt{3(11b^2-8b-16)}\in\N.
\]
So $11b^2-8b-16=3m^2$ for some $m\in\N$. Writing $b=2r$, as before, we have that $m$ is even, say $m=2\ell$. 
Thus $11r^2-4r-4=3\ell^2$. Note that as $a=2q=4$, we have $q=2$. So $r$ must be odd, since $q,r$ are coprime by Theorem~\ref{T:345}.
Let $r=2v+1$. Thus $44v^2+36 v + 3 =3\ell^2$. Hence $v$ is divisible by 3, say $v=3w$. So $b=2r=2(2v+1)=2(6w+1)=12w+2$, as in Example~\ref{E:verts4}. It follows that
$132 w^2+ 36 w+1=\ell^2$,
which is the same equation as \eqref{E:notPell}. Consequently, $b$ is one of the numbers $b_n$ of Example~\ref{E:verts4}. It follows that, up to a Euclidean transformation, the $OABC$ is one of the ELEPs of Example~\ref{E:verts4}.

Finally, suppose that the ELEP $OABC$ has a vertical side.
By translating and reflecting in the $x$ and/or $y$ axes if necessary, we may assume that the vertical side lies on the positive $y$-axis, starting at the origin $0$, and that the other side starting at $0$ lies in the 2nd quadrant. Therefore, using complex numbers, we consider an ELEP with vertices 
$A=a\sqrt3i,B=-x+(a+y)\sqrt3i,C=-x+y\sqrt3i, O=0$, where  $a,x,y\in\N$ by Remark~\ref{R:integers}. Then $OA$ has length $a\sqrt3$. Let $b\sqrt3$ denote the length of  $OC$. In particular, we have 
\begin{equation}\label{E:OA2}
3b^2=x^2+3y^2.
\end{equation}

The height from $C$ (to the $y$-axis) is $h:=x$, which is an integer. So, by Remark~\ref{R:widths},
if  $h=h_s$, then $OABC$ is the root parallelogram and $\{a,b\}=\{2,4\}$. So we may assume that $h=h_l$, in which case, by Remark~\ref{R:widths}, 
 $2b=ak$, for some positive integer $k$. Let $a=2q,b=2r$, as before. Then $2b=ak$ gives $2r=qk$. By Theorem~\ref{T:345}, $q,r$ are coprime, so $q=1$ or $2$; that is, $a=2$ or $4$. This completes the proof of the theorem.
\end{proof}



\bibliographystyle{amsplain}

\begin{thebibliography}{}

\bibitem{AD}
ABDELALIM, S.---DYANI, H.:
\textit{The solution of the diophantine equation
$x^2 + 3y^2 = z^2$},
{Int. J. Algebra} {\textbf  8} (2014), no.~15, 729 -- 732.

\bibitem{ACLEQI}
AEBI, C.---CAIRNS, G.:
\textit{Lattice equable quadrilaterals I: parallelograms}, L'Enseign.~Math.  \textbf{67} (2021), no.~3/4,  369--401.

\bibitem{ACLEQII}
AEBI, C.---CAIRNS, G.:
\textit{Lattice equable quadrilaterals II: kites, trapezoids and cyclic quadrilaterals},  Int.~J.~Geom.
\textbf{11} (2022), no.~2,  5--27.

\bibitem{ACLEQIII}
AEBI, C.---CAIRNS, G.:
\textit{Lattice equable quadrilaterals III: tangential and extangential cases}, Integers \textbf{23} (2023), A48 (107 pages).


\bibitem{ACMon}
AEBI, C.---CAIRNS, G.:
\textit{Following in Yiu's Footsteps but on the Eisenstein Lattice},   to appear in  Amer. Math. Monthly,
Preprint available at \url{http://arxiv.org/abs/2309.13551}.

\bibitem{ACetel}
AEBI, C.---CAIRNS, G.:
\textit{Equable triangles on the Eisenstein lattice},   to appear in Math. Gazette, Mar. 2025.
Preprint available at \url{https://arxiv.org/abs/2309.04476}.

\bibitem{AClte}
AEBI, C.---CAIRNS, G.:
 \textit{Less than Equable Triangles on the Eisenstein lattice},   to appear in Math. Gazette, July 2025.
Preprint available at \url{https://arxiv.org/abs/2312.10866}.


\bibitem{Alpern} 
ALPERN, D.:
\textit{Generic two integer variable equation solver},  \url{https://www.alpertron.com.ar/QUAD.HTM}, accessed 6 September 2023. 

\bibitem{BU}  
BARAGAR, A.---UMEDA, K.:
 \textit{The asymptotic growth of integer solutions to the
              {R}osenberger equations},
{Bull. Austral. Math. Soc.},
{\textbf 69} (2004),
no.~3,
{481--497}.

\bibitem{Barb} 
BARBEAU, E. J.:
\textit{Pell's equation},
Springer-Verlag, New York, 2003.

\bibitem{De} 
DEEMER, B.: 
\textit{A Recurrence Formula Solution to $dy^2+1=x^2$},
Math. Mag. {\textbf 32} (1958), no.~1, 37--40.

\bibitem{OW} 
OSTERMANN, A.---WANNER, G.:
\textit{Geometry by its History},
 {Springer, Heidelberg},
 {2012}.

\bibitem{Ro} 
ROSENBERGER, G.: 
\textit{\"{U}ber die diophantische {G}leichung {$ax^{2}+by^{2}+cz^{2}=dxyz$}},
{J. Reine Angew. Math.}
{\textbf 305} (1979), 122--125.
 
\bibitem{OEIS} 
SLOANE, N. J. A.:
\textit{The On-line Encyclopedia of Integer Sequences}, \url{https://oeis.org}.

\bibitem{art} 
WANG, V.:
\textit{The Art of problem Solving}, \url{https://artofproblemsolving.com/community/c1461h1035155}.

\bibitem{Yuan} YUAN, Q.:
\textit{Annoying Precision}, \url{https://qchu.wordpress.com/2009/07/02/square-roots-have-no-unexpected-linear-relationships/}.

\bibitem{Zim} 
ZIMHONI, N.:
\textit{A forest of Eisensteinian triangles},
{Amer. Math. Monthly} {\textbf 127} (2020), no.~7, 629--637. 

\end{thebibliography}
{}

\end{document}